\documentclass[12pt]{article}

\usepackage{amsmath, amsthm, amssymb, stmaryrd}
\usepackage[all]{xy}
\usepackage{fullpage}
\usepackage{titlesec}
\usepackage{mathrsfs}
\usepackage{amsbsy}
\usepackage{turnstile}
\usepackage{bbm}
\usepackage{yhmath}
\usepackage{tensor}
\usepackage{graphics}
\usepackage{enumitem}
\usepackage{calligra}
\usepackage{verbatim}

\usepackage{setspace}

\usepackage[pdftex,bookmarks=true]{hyperref}
\usepackage[svgnames]{xcolor}
\hypersetup{
    linkbordercolor = {PaleTurquoise},
    citebordercolor = {PaleGreen},
}

\makeatletter
\newcommand*{\@old@slash}{}\let\@old@slash\slash
\def\slash{\relax\ifmmode\delimiter"502F30E\mathopen{}\else\@old@slash\fi}
\makeatother

\titleformat{\section}{\normalsize\bfseries}{\thesection}{1em}{}
\titleformat{\subsection}{\normalsize\bfseries}{\thesubsection}{1em}{}

\numberwithin{equation}{subsection}

\theoremstyle{plain}

\newtheorem{PropSub}[subsection]{Proposition}
\newtheorem{LemSub}[subsection]{Lemma}
\newtheorem{CorSub}[subsection]{Corollary}
\newtheorem{ThmSub}[subsection]{Theorem}

\theoremstyle{definition}

\newtheorem{DefSub}[subsection]{Definition}
\newtheorem{ExaSub}[subsection]{Example}
\newtheorem{RemSub}[subsection]{Remark}
\newtheorem{ParSub}[subsection]{}

\newcommand*{\emptybox}{\leavevmode\hbox{}}

\newdir{ >}{{}*!/-10pt/@{>}}

\DeclareMathAlphabet{\mathpzc}{OT1}{pzc}{m}{it}
\DeclareMathAlphabet{\mathcalligra}{T1}{calligra}{m}{n}

\newcommand{\bref}[1]{\textnormal{\ref{#1}}}
\newcommand{\pbref}[1]{\textnormal{(\ref{#1})}}

\newcommand{\boldref}[1]{\textbf{\ref{#1}}}

\newcommand{\tinyplus}{\scalebox{0.6}{$\uparrow$}}
\newcommand{\tinyminus}{\scalebox{0.6}{$\downarrow$}}

\newcommand{\cO}{\ensuremath{\mathcal{O}}}

\newcommand{\A}{\ensuremath{\mathscr{A}}}
\newcommand{\B}{\ensuremath{\mathscr{B}}}
\newcommand{\C}{\ensuremath{\mathscr{C}}}

\newcommand{\F}{\ensuremath{\mathscr{F}}}
\newcommand{\G}{\ensuremath{\mathscr{G}}}
\newcommand{\sH}{\ensuremath{\mathscr{H}}}

\newcommand{\J}{\ensuremath{\mathscr{J}}}
\newcommand{\K}{\ensuremath{\mathscr{K}}}
\newcommand{\sL}{\ensuremath{\mathscr{L}}}
\newcommand{\M}{\ensuremath{\mathscr{M}}}
\newcommand{\N}{\ensuremath{\mathscr{N}}}

\newcommand{\T}{\ensuremath{\mathscr{T}}}
\newcommand{\U}{\ensuremath{\mathscr{U}}}
\newcommand{\V}{\ensuremath{\mathscr{V}}}
\newcommand{\W}{\ensuremath{\mathscr{W}}}

\newcommand{\shortunderline}[1]{\underline{#1\mkern-5mu}\mkern5mu}

\newcommand{\aeV}{\ensuremath{\shortunderline{\mathscr{V}}}}

\newcommand{\uM}{\ensuremath{\shortunderline{\mathscr{M}}}}
\newcommand{\uN}{\ensuremath{\underline{\mathscr{N}\mkern-9mu}\mkern8mu}}
\newcommand{\uU}{\ensuremath{\mkern3mu\underline{\mkern-3mu\mathscr{U}\mkern-5mu}\mkern4mu}}
\newcommand{\uV}{\ensuremath{\mkern2mu\underline{\mkern-2mu\mathscr{V}\mkern-6mu}\mkern5mu}}
\newcommand{\uW}{\ensuremath{\underline{\mathscr{W}\mkern-7mu}\mkern6mu}}

\newcommand{\CAT}{\ensuremath{\operatorname{\textnormal{\text{CAT}}}}}
\newcommand{\eCAT}[1]{\ensuremath{#1\textnormal{-\text{CAT}}}}
\newcommand{\VCAT}{\ensuremath{\V\textnormal{-\text{CAT}}}}
\newcommand{\WCAT}{\ensuremath{\W\textnormal{-\text{CAT}}}}

\newcommand{\eMCAT}[1]{\ensuremath{#1\textnormal{-\text{MCAT}}}}
\newcommand{\eSMCAT}[1]{\ensuremath{#1\textnormal{-\text{SMCAT}}}}
\newcommand{\eSMCCAT}[1]{\ensuremath{#1\textnormal{-\text{SMCCAT}}}}

\newcommand{\VMCAT}{\ensuremath{\eMCAT{\V}}}
\newcommand{\VSMCAT}{\ensuremath{\eSMCAT{\V}}}
\newcommand{\VSMCCAT}{\ensuremath{\eSMCCAT{\V}}}

\newcommand{\WSMCAT}{\ensuremath{\eSMCAT{\W}}}
\newcommand{\WSMCCAT}{\ensuremath{\eSMCCAT{\W}}}

\newcommand{\ONE}{\ensuremath{\textnormal{\text{1}}}}
\newcommand{\bI}{\ensuremath{\textnormal{\text{I}}}}

\newcommand{\MCAT}{\ensuremath{\operatorname{\textnormal{\text{MCAT}}}}}
\newcommand{\SMCAT}{\ensuremath{\operatorname{\textnormal{\text{SMCAT}}}}}
\newcommand{\SMCCAT}{\ensuremath{\operatorname{\textnormal{\text{SMCCAT}}}}}

\newcommand{\ECAT}{\ensuremath{\int \eCAT{(-)}}}
\newcommand{\ESMCCAT}{\ensuremath{\int \eSMCCAT{(-)}}}
\newcommand{\ENRSMCCAT}{\ESMCCAT}
\newcommand{\TWOCAT}{\ensuremath{\operatorname{\textnormal{\text{2CAT}}}}}
\newcommand{\MTWOCAT}{\ensuremath{\operatorname{\textnormal{\text{M2CAT}}}}}
\newcommand{\SMTWOCAT}{\ensuremath{\operatorname{\textnormal{\text{SM2CAT}}}}}

\newcommand{\ob}{\ensuremath{\operatorname{\textnormal{\textsf{Ob}}}}}

\newcommand{\Dom}{\ensuremath{\operatorname{\textnormal{\textsf{Dom}}}}}

\newcommand{\Enr}{\ensuremath{\operatorname{\textnormal{\textsf{Enr}}}}}

\newcommand{\Shv}{\ensuremath{\textnormal{\textnormal{Shv}}}}

\newcommand{\Ev}{\ensuremath{\textnormal{\textsf{Ev}}}}

\newcommand{\SET}{\ensuremath{\operatorname{\textnormal{\text{SET}}}}}

\newcommand{\Mod}[1]{\ensuremath{#1\textnormal{-\text{Mod}}}}
\newcommand{\Sch}{\ensuremath{\operatorname{\textnormal{\text{Sch}}}}}

\newcommand{\PsMon}{\ensuremath{\operatorname{\textnormal{\text{PsMon}}}}}
\newcommand{\SymPsMon}{\ensuremath{\operatorname{\textnormal{\text{SymPsMon}}}}}
\newcommand{\WMon}{\ensuremath{\operatorname{\textnormal{\text{WMon}}}}}

\newcommand{\op}{\ensuremath{\textnormal{op}}}
\newcommand{\co}{\ensuremath{\textnormal{co}}}
\newcommand{\coop}{\ensuremath{\textnormal{coop}}}

\newcommand{\pushoutcorner}{\ar@{}[dr]|(.3)\ulcorner}
\newcommand{\pullbackcorner}{\ar@{}[dr]|(.3)\lrcorner}

\newcommand{\widegrave}[1]{\grave{\wideparen{#1}}}

\newcommand{\cmt}[1]{}

\begin{document}

\author{\normalsize  Rory B. B. Lucyshyn-Wright\thanks{The author gratefully acknowledges financial support in the form of an NSERC Postdoctoral Fellowship.}\let\thefootnote\relax
\\
\small University of Cambridge, Wilberforce Road, Cambridge, UK}

\title{\large \textbf{Relative symmetric monoidal closed categories I:}\\ \large\textbf{Autoenrichment and change of base}\\
       \emptybox\\
       \normalsize{\textsl{Dedicated to G. M. Kelly on the occasion of the fiftieth anniversary}\\\textsl{ of the La Jolla Conference on Categorical Algebra, 1965}}}

\date{}

\maketitle

\abstract{Symmetric monoidal closed categories may be related to one another not only by the functors between them but also by enrichment of one in another, and it was known to G.~M. Kelly in the 1960s that there is a very close connection between these phenomena.  In this first part of a two-part series on this subject, we show that the assignment to each symmetric monoidal closed category $\V$ its associated $\V$-enriched category $\uV$ extends to a 2-functor valued in an op-2-fibred 2-category of symmetric monoidal closed categories enriched over various bases.  For a fixed $\V$, we show that this induces a 2\nobreak\mbox{-}\nobreak\hspace{0pt}functorial passage from symmetric monoidal closed categories \textit{over} $\V$ (i.e., equipped with a morphism to $\V$) to symmetric monoidal closed $\V$-categories over $\uV$.  As a consequence, we find that the enriched adjunction determined a symmetric monoidal closed adjunction can be obtained by applying a 2-functor and, consequently, is an adjunction in the 2-category of symmetric monoidal closed $\V$-categories.
}

\section{Introduction} \label{sec:intro}

In Grothendieck's approach to algebraic geometry, one studies spaces over a given base space $S$, and particularly schemes (or algebraic spaces, or stacks) over a base scheme $S$, thus working within the slice category $\Sch\slash S$, and then \textit{change of base} along a morphism becomes important.  This \textit{relative point of view}, coupled with Grothendieck's practice of studying a space by means of the category of sheaves thereupon, also led to a relative point of view on categories, from which sprouted the notions of \textit{fibred category} \cite{Gr:Fibred} and \textit{indexed category}, and the study of toposes over a base.

Another distinct notion of \textit{relative} category is the concept of \textit{enriched category}, which arose with the observation that in many categories $\C$ the set $\C(C,D)$ of all morphisms between objects $C$ and $D$ of $\C$ is merely the `underlying set' of some more substantially structured mathematical object, such as an abelian group, simplicial set, or cochain complex, and so is but a pale shadow of an object $\C(C,D)$ of another category $\V$.  The notion of \textit{$\V$-enriched category} developed through the work of several authors in the first half of the 1960s, and a comprehensive basis for the study of $\V$-categories was expounded by Eilenberg and Kelly in the proceedings of the 1965 Conference on Categorical Algebra at La Jolla \cite{EiKe}.  The latter work made it clear that the theory of $\V$-categories gains considerable depth when $\V$ is \textit{itself} $\V$-enriched, so that $\V$ is a \textit{closed category}.  We denote the resulting $\V$-category by $\uV$ and call it the \textit{autoenrichment} of $\V$.

For example, given a scheme $S$, the following are closed categories and so are enriched in themselves; see, e.g., \cite[3.5.2]{Lip} for (ii) and (iii).
\begin{enumerate}
\item[(i)] The category of sheaves $\Shv(S)$ for, say, the small Zariski site of $S$.
\item[(ii)] The category $\Mod{\cO_S}$ of (sheaves of) $\cO_S$-modules.
\item[(iii)] The derived category $\mathop{\textnormal{D}}(\Mod{\cO_S})$.
\end{enumerate}
While the \textit{hom} operation on $\cO_S$-modules witnesses that $\Mod{\cO_S}$ is closed, the related \textit{tensor product} of $\cO_S$-modules makes this category \textit{symmetric monoidal closed}.  As a topos, $\Shv(S)$ is also symmetric monoidal closed, but with the cartesian product of sheaves $\times$ playing the role of monoidal product.  Now letting $X$ be a scheme over $S$, via $g:X \rightarrow S$, and taking
\begin{enumerate}
\item[(i)] $\M = \Shv(X)$, $\V = \Shv(S)$,
\item[(ii)] $\M = \Mod{\cO_X}$, $\V = \Mod{\cO_S}$, or
\item[(iii)] $\M = \mathop{\textnormal{D}}(\Mod{\cO_X})$, $\V = \mathop{\textnormal{D}}(\Mod{\cO_S})$
\end{enumerate}
we obtain in each case a morphism of symmetric monoidal closed categories
$$G:\M \longrightarrow \V\;,$$
i.e., a \textit{symmetric monoidal (closed) functor}, namely the direct image functor $g_*$, which participates with the corresponding inverse image functor $F = g^*$ in a \textit{symmetric monoidal (closed) adjunction} $F \dashv G:\M \rightarrow \V$ \pbref{exa:smc_adj_sheaves}.  As was observed by Eilenberg and Kelly in a general setting \cite[I 6.6]{EiKe}, $\M$ therefore acquires the structure of a $\V$-enriched category, and $G$ becomes a $\V$-enriched functor; moreover the adjunction becomes $\V$-enriched \cite[5.1]{Ke:EnrAdj}.  In fact, all the symmetric monoidal closed structure involved becomes entirely $\V$-functorial and $\V$-natural, so that $\M$ is a \textit{symmetric monoidal closed $\V$-category} and $F \dashv G$ a \textit{symmetric monoidal closed $\V$-adjunction} \pbref{thm:enr_smcadj}.

In Part I of the present work, we study several aspects of the rather subtle 2\nobreak\mbox{-}\nobreak\hspace{0pt}functoriality of the autoenrichment assignment $\V \mapsto \uV$, as well as the 2-functoriality of related processes by which a symmetric monoidal closed category may acquire enrichment in another.  The present paper shall provide the basis for a 2-functorial study in Part II of the relations between the following three notions of \textit{relative} symmetric monoidal closed category, and variations thereupon:
\begin{enumerate}
\item Symmetric monoidal closed categories $\M$ equipped with a morphism $\M \rightarrow \V$.
\item Symmetric monoidal closed $\V$-categories $\M$.
\item The fully $\V$-enriched analogue of 1.
\end{enumerate}
In Part II we shall establish several equivalences of \mbox{2-categories} that provide elaborations and variations on a seemingly unpublished result of G.~M. Kelly in this regard \cite{St:CatsPost}, and we shall accord particular attention to those cases in which the morphism $\M \rightarrow \V$ is equipped with a left adjoint.  The categories of algebras of suitable \textit{commutative monads} on $\V$ constitute a broad class of examples of such categories $\M$ \cite{Kock:ClsdCatsGenCommMnds}, \cite[5.5.4, 5.5.7]{Lu:PhD}.

In order to understand the sense in which autoenrichment is 2-functorial, let us return to the above example of a scheme $X$ over a base scheme $S$.  The direct image functor $G:\M \rightarrow \V$ determines a \textit{change-of-base} 2-functor
$$G_*:\eCAT{\M} \longrightarrow \VCAT$$
from the 2-category of $\M$-categories to the 2-category of $\V$-categories (\cite[I 10.5]{EiKe}).  By applying $G_*$ to the autoenrichment $\uM$ we find that $\M$ is $\V$-enriched; for example, when $\M = \Mod{\cO_X}$, the $\V$-valued hom-objects for $G_*\uM$ are the $\cO_S$-modules
$$(G_*\uM)(M,N) = g_*(\mathscr{H}\text{\kern -3pt {\calligra\large om}}_{{}_{\cO_X}}\kern -3pt(M,N))\;\;\;\;(M,N \in \Mod{\cO_X}).$$
\cmt{Typesetting of sheaf hom is due to Daniel Miller (http://tex.stackexchange.com/questions/141434/how-to-type-sheaf-hom).}Further, $G$ determines a $\V$-functor\footnote{In \cite{EiKe} the notation $\widehat{G}$ is used, but we shall see a reason for the present notation $\grave{G}$ in \bref{thm:enr_smcadj} and \bref{rem:enr_smcadj}.}
$$\grave{G}:G_*\uM \longrightarrow \uV\;.$$
The same is true for an arbitrary symmetric monoidal closed functor $G$ (as was observed by Eilenberg and Kelly), and the assignment $G \mapsto \grave{G}$ can be seen as functorial as soon as we construe $\grave{G}$ not as a mere $\V$-functor but as a 1-cell over $G$ in an \textit{op-2-fibred} 2-category of categories enriched over various bases, as follows.  Eilenberg and Kelly showed that the assignment $G \mapsto G_*$ gives rise to a 2-functor
$$\eCAT{(-)}:\SMCCAT \longrightarrow \TWOCAT,\;\;\;\;\V \mapsto \VCAT$$
from the 2-category $\SMCCAT$ of symmetric monoidal closed categories to the 2\nobreak\mbox{-}\nobreak\hspace{0pt}category $\TWOCAT$ of 2-categories.  Through a variation on Grothendieck's construction of the fibred category determined by a pseudofunctor, work of Bakovi\'c and of Buckley \cite{Buck} on 2-fibrations entails that the latter 2-functor determines an \textit{op-2-fibration}
$$\ECAT \longrightarrow \SMCCAT\;,\;\;\;\;(\V,\C) \mapsto \V$$
where $\ECAT$ is a 2-category whose objects are pairs $(\V,\C)$ consisting of a symmetric monoidal closed category $\V$ and a $\V$-category $\C$; see \bref{par:2groth_constr}.

As we shall see, the autoenrichment then extends to a 2-functor
$$\SMCCAT \longrightarrow \ECAT\;,\;\;\;\;\V \mapsto (\V,\uV)$$
valued in the 2-category $\ECAT$ of categories enriched over various bases.

But the autoenrichment $\uV$ of a symmetric monoidal closed category $\V$ is not merely an enriched category but in fact a \textit{symmetric monoidal closed $\V$-category}.  Therefore, we show that the autoenrichment in fact yields a 2-functor
$$\SMCCAT \longrightarrow \ESMCCAT\;,\;\;\;\;\V \mapsto \uV$$
valued in a 2-category $\ESMCCAT$ of symmetric monoidal closed $\V$-categories for various bases $\V$.  In order to define the latter 2-category, we must first study change of base for symmetric monoidal closed $\V$-categories.  Using general theory of pseudomonoids, we show that there is a 2-functor
$$\eSMCCAT{(-)}:\SMCCAT \longrightarrow \TWOCAT\;,\;\;\;\;\V \mapsto \VSMCCAT$$
sending $\V$ to the 2-category $\VSMCCAT$ of symmetric monoidal closed $\V$-categories.  By applying the Bakovi\'c-Buckley-Grothendieck construction to this 2-functor we therefore obtain an op-2-fibration
$$\ESMCCAT \longrightarrow \SMCCAT\;,\;\;\;\;(\V,\M) \mapsto \V$$
where $\ESMCCAT$ is a 2-category whose objects $(\V,\M)$ consist of a symmetric monoidal closed category $\V$ and a symmetric monoidal closed $\V$-category $\M$.  

Since the description of 2-cells and pasting in $\ESMCCAT$ is quite complicated, the 2-functoriality of the autoenrichment is therefore a correspondingly subtle matter.  Several useful lemmas, such as the $\V$-enriched monoidality of $\grave{G}$, are enfolded within the resulting 2-functorial autoenrichment.

For a \textit{fixed} symmetric monoidal closed category $\V$, we show that the autoenrichment induces a 2-functorial passage from symmetric monoidal closed categories $\M$ \textit{over} $\V$ (i.e., equipped with a morphism $\M \rightarrow \V$) to symmetric monoidal closed $\V$-categories over the autoenrichment $\uV$.  Explicitly, we obtain a 2-functor
$$\Enr_\V:\SMCCAT\sslash\V \longrightarrow \VSMCCAT \sslash \uV$$
between the \textit{lax slices}\footnote{We use the term \textit{lax slice} here for what some authors quite rightly call the \textit{colax slice 2-category} \pbref{def:lax_slice}.} over $\V$ in $\SMCCAT$ and $\VSMCCAT$ respectively, sending each object $G:\M \rightarrow \V$ of the former 2-category to the object $\grave{G}:G_*(\uM) \rightarrow \uV$ of the latter.   For example, given a scheme $X$ over $S$, the category $\M = \Mod{\cO_X}$ of $\cO_X$-modules is a symmetric monoidal closed category over $\V = \Mod{\cO_S}$ and, via this 2-functor, is moreover a symmetric monoidal closed $\V$-category over $\uV$.

Given a symmetric monoidal closed \textit{adjunction} $F \dashv G:\M \rightarrow \V$, Kelly showed that the associated $\V$-functor $\grave{G}$ has a left $\V$-adjoint \cite[5.1]{Ke:EnrAdj}.  Herein we show that an associated symmetric monoidal closed $\V$-adjunction
$$\acute{F} \dashv \grave{G}:G_*(\uM) \longrightarrow \uV$$
can be obtained simply by applying the composite 2-functor
$$
\xymatrix{
\SMCCAT\sslash\V \ar[r]^(.45){\Enr_\V} & \VSMCCAT \sslash \uV \ar[r] & \VSMCCAT
}
$$
to the adjunction $(F,\eta) \dashv (G,1_G):(\M,G) \rightarrow (\V,1_\V)$ in $\SMCCAT\sslash \V$.  For example, for a scheme $X$ over $S$, the adjoint pair $g^* \dashv g_*:\Mod{\cO_X} \rightarrow \Mod{\cO_S}$ carries the structure of a symmetric monoidal closed $\V$-adjunction for $\V = \Mod{\cO_S}$.

\section{Split 2-fibrations}\label{sec:2fibr}

Building on earlier work of Hermida and Bakovi\'c, Buckley \cite{Buck} defined a notion of \mbox{\textit{2-fibration}} and established a suitable analogue of Grothendieck's correspondence between fibrations and certain pseudofunctors \cite{Gr:Fibred}.  One has also the dual notions of op-2-fibration, co-2-fibration, and coop-2-fibration \cite[2.2.14]{Buck}.  A 2-functor $P:\F \rightarrow \K$ is said to be an \textit{op-2-fibration} if the 2-functor $P^\op:\F^\op \rightarrow \K^\op$ is a 2-fibration.  Since here we shall make detailed use of the notion of op-2-fibration, we now explicitly state its definition in terms of the notions of \textit{cocartesian 1-cell} and \textit{cartesian 2-cell}.

\begin{DefSub}
Let $P:\F \rightarrow \K$ be a 2-functor.
\begin{enumerate}
\item Given a 1-cell $f:A \rightarrow B$ in $\F$, an \textit{extension problem} for $f$ (relative to $P$) is a triple $(f,\alpha,\beta)$ in which $\alpha:g \Rightarrow h:A \rightarrow C$ is a 2-cell in $\F$ and $\beta:k \Rightarrow \ell:PB \rightarrow PC$ is a 2-cell in $\K$ such that the equation of 2-cells $\beta \circ Pf = P\alpha$ holds\footnote{We interpret such an equation of 2-cells as asserting also that the domain (resp. codomain) 1-cells are equal, i.e. that $k \circ Pf = Pg$ and $\ell \circ Pf = Ph$.}.  A \textit{solution} to $(f,\alpha,\beta)$ is a 2-cell $\beta':k' \Rightarrow \ell':B \rightarrow C$ in $\F$ such that $\beta' \circ f = \alpha$ and $P\beta' = \beta$.
\item A 1-cell $f$ in $\F$ is \textit{cocartesian} (with respect to $P$) if every extension problem for $f$ has a unique solution.
\item Given a 2-cell $\phi:f \Rightarrow g:A \rightarrow B$ in $\F$, a \textit{lifting problem} for $\phi$ (relative to $P$) is a triple $(\phi,\gamma,\kappa)$ where $\gamma:h \Rightarrow g$ is a 2-cell in $\F$, $\kappa:Ph \Rightarrow Pf$ is a 2-cell in $\K$, and $P\phi \cdot \kappa = P\gamma$.  A \textit{solution} to $(\phi,\gamma,\kappa)$ is a 2-cell $\kappa':h \Rightarrow f$ in $\F$ such that $\phi \cdot \kappa' = \gamma$ and $P\kappa' = \kappa$.
\item A 2-cell $\phi:f \Rightarrow g:A \rightarrow B$ in $\F$ is \textit{cartesian} if every lifting problem for $\phi$ has a unique solution, i.e. if $\phi$ is a cartesian arrow for the ordinary functor $P_{AB}:\F(A,B) \rightarrow \K(PA,PB)$.
\item $P$ is an \textit{op-2-fibration} if (i) for every 1-cell $k:K \rightarrow L$ in $\K$ and every object $A$ over $K$ in $\F$ (i.e., with $PA = K$), there exists a cocartesian 1-cell $A \rightarrow B$ over $k$ in $\F$, (ii) for every 2-cell $\kappa:k \Rightarrow \ell:K \rightarrow L$ in $\K$ and every 1-cell $g:A \rightarrow B$ over $\ell$ in $\F$, there exists a cartesian 2-cell $f \Rightarrow g:A \rightarrow B$ over $\kappa$ in $\F$, and (iii) the cartesian 2-cells are closed under whiskering with arbitrary 1-cells.
\item $P$ is a \textit{cloven} op-2-fibration if $P$ is an op-2-fibration equipped with a specified choice of cocartesian 1-cells and cartesian 2-cells.  Given data as in 5(i) and 5(ii) above, we write the associated cocartesian 1-cell as $\psi(k,A):A \rightarrow k_*(A)$ and the associated cartesian 2-cell as $\varphi(\kappa,g):\kappa^*(g) \Rightarrow g$.  We say that a 1-cell $f:A \rightarrow B$ in $\F$ is a \textit{designated cocartesian 1-cell} if $f = \psi(Pf,A)$, and we say that a 2-cell $\phi:f \Rightarrow g$ in $\F$ is a \textit{designated cartesian 2-cell} if $\phi = \varphi(P\phi,g)$.
\item $P$ is a \textit{split} op-2-fibration if $P$ is a cloven op-2-fibration such that the designated cocartesian 1-cells are closed under composition, the designated cartesian 2-cells are closed under vertical composition and under whiskering with arbitrary 1-cells, and the designated cocartesian (resp. cartesian) 1-cell (resp. 2-cell) associated to an identity 1-cell (resp. 2-cell) is again an identity 1-cell (resp. 2-cell).
\end{enumerate}
\end{DefSub}

\begin{ParSub}
If $f:A \rightarrow B$ is a cocartesian 1-cell with respect to $P:\F \rightarrow \K$, then $f$ is a cocartesian arrow with respect to the underlying ordinary functor of $P$.  Indeed, given 1-cells $g:A \rightarrow C$ in $\F$ and $k:PB \rightarrow PC$ in $\K$ with $k \circ Pf = Pg$, the extension problem $(f,1_g,1_k)$ for $f$ has a unique solution, and one readily finds that this solution must in fact be the identity 2-cell for a unique 1-cell $k':B \rightarrow C$ with $k' 
\circ f = g$ and $Pk' = k$.  We write such an extension problem as simply $(f,g,k)$.
\end{ParSub}

\begin{ParSub}[The Bakovi\'c-Buckley-Grothendieck construction]\label{par:2groth_constr}
By Buckley's work \cite{Buck}, there is a correspondence between
\begin{enumerate}
\item[(1)] 2-functors $\Phi:\K \rightarrow \TWOCAT$, and
\item[(2)] split op-2-fibrations $P:\F \rightarrow \K$ 
\end{enumerate}
where $\TWOCAT$ denotes the 2-category of 2-categories.  In fact, \cite[2.2.11]{Buck} establishes an equivalence of 3-categories relating $\TWOCAT$-valued 2-functors on $\K^{\co\op}$ to 2-fibrations $\F \rightarrow \K$, from which a similar equivalence between (1) and (2) then follows, through an adaptation that we shall now discuss.  Here we shall require only the passage from (1) to (2), which is a variation on Grothendieck's construction of the fibration associated with a $\CAT$-valued pseudofunctor \cite{Gr:Fibred}.  Whereas Buckley \cite[2.2.1]{Buck} constructs the 2-fibration associated to a 2-functor $\K^\coop \rightarrow \TWOCAT$, if instead given a 2-functor $\Phi:\K \rightarrow \TWOCAT$ one can apply Buckley's construction to the composite
\begin{equation}\label{eq:comp_w_op}(\K^\op)^{\coop} = \K^\co \xrightarrow{\Phi^\co} \TWOCAT^\co \xrightarrow{(-)^\op} \TWOCAT\end{equation}
in order to obtain a 2-fibration $Q:\G \rightarrow \K^\op$ and thus an op-2-fibration $P:\int\Phi \rightarrow \K$ by taking $\int\Phi = \G^\op$ and $P = Q^\op$.    The fibre 2-category $P^{-1}(K)$ over each object $K \in \K$ is then isomorphic to $\Phi K$.

Explicitly, the objects of $\int\Phi$ are pairs $A = (A^{\tinyminus},A^{\tinyplus})$ consisting of an object $A^{\tinyminus} \in \K$ and an object $A^{\tinyplus} \in \Phi(A^{\tinyminus})$.  The 1-cells $A \rightarrow B$ in $\int\Phi$ are pairs $f = (f^{\tinyminus},f^{\tinyplus})$ consisting of a 1-cell $f^{\tinyminus}:A^{\tinyminus} \rightarrow B^{\tinyminus}$ in $\K$ and a 1-cell $f^{\tinyplus}:f^{\tinyminus}_*(A^{\tinyplus}) \rightarrow B^{\tinyplus}$ in $\Phi(B^{\tinyminus})$, where we write $f^{\tinyminus}_* = \Phi(f^{\tinyminus}):\Phi(A^{\tinyminus}) \rightarrow \Phi(B^{\tinyminus})$.  The composite of 1-cells
$$A \xrightarrow{f} B \xrightarrow{g} C$$
in $\int\Phi$ is the pair consisting of the composite $g^{\tinyminus}f^{\tinyminus}:A^{\tinyminus} \rightarrow C^{\tinyminus}$ in $\K$ together with the composite 
$$(g^{\tinyminus}f^{\tinyminus})_*(A^{\tinyplus}) = g^{\tinyminus}_*f^{\tinyminus}_*(A^{\tinyplus}) \xrightarrow{g^{\tinyminus}_*(f^{\tinyplus})} g^{\tinyminus}_*(B^{\tinyplus}) \xrightarrow{g^{\tinyplus}} C^{\tinyplus}$$
in $\Phi(C^{\tinyminus})$.

A 2-cell $f \Rightarrow g:A \rightarrow B$ in $\int\Phi$ is a pair $\alpha = (\alpha^{\tinyminus},\alpha^{\tinyplus})$ consisting of a 2-cell $\alpha^{\tinyminus}:f^{\tinyminus} \Rightarrow g^{\tinyminus}$ in $\K$ together with a 2-cell
\begin{equation}\label{eq:gr_const_2cell}
\xymatrix{
f^{\tinyminus}_*(A^{\tinyplus}) \ar[d]_{\alpha^{\tinyminus}_*A^{\tinyplus}} \ar[drr]^{f^{\tinyplus}}="s1" & & \\
g^{\tinyminus}_*(A^{\tinyplus}) \ar[rr]_{g^{\tinyplus}} \ar@{}"s1";{}|(.4){}="s2"|(.6){}="t2" & & B^{\tinyplus}
  \ar@{=>}"s2";"t2"_{\alpha^{\tinyplus}}
}
\end{equation}
in $\Phi(B^{\tinyminus})$, where we write $\alpha^{\tinyminus}_* = \Phi(\alpha^{\tinyminus}):f^{\tinyminus}_* \Rightarrow g^{\tinyminus}_*$.  Vertical composition of 2-cells in $\int\Phi$ is given by taking the vertical composite of the associated 2-cells in $\K$ and pasting the 2-cells \eqref{eq:gr_const_2cell} in the relevant fibre 2-category.  Whiskering in $\int\Phi$ is given as follows.  Given
$$
\xymatrix{
A \ar[r]^{f} & B \ar@/^2ex/[rrr]^g="s1" \ar@/_2ex/[rrr]_h="t1" & & & C \ar[r]^u & D
  \ar@{}"s1";"t1"|(.4){}="s2"|(.6){}="t2"
  \ar@{=>}"s2";"t2"^\alpha
}
$$
in $\int\Phi$, the associated 2-cell $\alpha f:gf \Rightarrow hf$ in $\int\Phi$ consists of the 2-cell $\alpha^{\tinyminus} f^{\tinyminus}:g^{\tinyminus}f^{\tinyminus} \Rightarrow h^{\tinyminus}f^{\tinyminus}$ in $\K$ together with the composite 2-cell
$$
\xymatrix{
g^{\tinyminus}_*f^{\tinyminus}_*(A^{\tinyplus}) \ar[d]_{(\alpha^{\tinyminus} f^{\tinyminus})_*A^{\tinyplus}\:=\:\alpha^{\tinyminus}_*f^{\tinyminus}_*(A^{\tinyplus})} \ar[r]^{g^{\tinyminus}_*(f^{\tinyplus})} & g^{\tinyminus}_*(B^{\tinyplus}) \ar@{}[dl]|(.45){}="s3"|(.55){}="t3" \ar[d]_{\alpha^{\tinyminus}_* B^{\tinyplus}} \ar[drr]^{g^{\tinyplus}}="s1" & & \\
h^{\tinyminus}_*f^{\tinyminus}_*(A^{\tinyplus}) \ar[r]_{h^{\tinyminus}_*(f^{\tinyplus})} & h^{\tinyminus}_*(B^{\tinyplus}) \ar[rr]_{h^{\tinyplus}} \ar@{}"s1";{}|(.4){}="s2"|(.6){}="t2" & & C^{\tinyplus}
  \ar@{=>}"s2";"t2"_{\alpha^{\tinyplus}}
  \ar@{=}"s3";"t3"
}
$$
in $\Phi(C^{\tinyminus})$, in which the indicated identity 2-cell is obtained by the naturality of $\alpha^{\tinyminus}_*$.  The associated 2-cell $u\alpha:ug \Rightarrow uh$ in $\int\Phi$ consists of the 2-cell $u^{\tinyminus}\alpha^{\tinyminus}:u^{\tinyminus}g^{\tinyminus} \Rightarrow u^{\tinyminus}h^{\tinyminus}$ in $\K$ together with the composite 2-cell
$$
\xymatrix{
u^{\tinyminus}_*g^{\tinyminus}_*(B^{\tinyplus}) \ar[d]_{(u^{\tinyminus}\alpha^{\tinyminus})_*B^{\tinyplus}\:=\:u^{\tinyminus}_*\alpha^{\tinyminus}_*B^{\tinyplus}} \ar[drr]^{u^{\tinyminus}_*(g^{\tinyplus})}="s1" & & & \\
u^{\tinyminus}_*h^{\tinyminus}_*(B^{\tinyplus}) \ar[rr]_{u^{\tinyminus}_*(h^{\tinyplus})} \ar@{}"s1";{}|(.4){}="s2"|(.6){}="t2" & & u^{\tinyminus}_*(C^{\tinyplus}) \ar[r]^{u^{\tinyplus}} & D^{\tinyplus}
  \ar@{=>}"s2";"t2"_{u_*(\alpha^{\tinyplus})}
}
$$
in $\Phi(D^{\tinyminus})$.

In $\int\Phi$, the designated cocartesian 1-cell $\psi(k,A):A \rightarrow k_*(A)$ associated to a 1-cell $k:K \rightarrow L$ in $\K$ and an object $A = (K,A^{\tinyplus})$ of $\int\Phi$ is defined as $(k,1_{k_*(A^{\tinyplus})}):(K,A^{\tinyplus}) \rightarrow (L,k_*(A^{\tinyplus}))$.  The designated cartesian 2-cell $\varphi(\kappa,g):\kappa^*(g) \Rightarrow g$ associated to a 2-cell $\kappa:k \Rightarrow \ell$ of $\K$ and a 1-cell $g:A \rightarrow B$ over $\ell$ in $\int\Phi$ is defined as $(\kappa,1):(k,g^{\tinyplus} \circ \kappa_*A^{\tinyplus}) \Rightarrow (\ell,g^{\tinyplus}) = g$.
\end{ParSub}

\section{Monoidal and enriched categories} \label{sec:enr_mon_func}

We shall employ the theory of monoidal categories, closed categories, and categories enriched in a monoidal category $\V$, as expounded in \cite{EiKe,Ke:Ba}.  Unless otherwise indicated, we shall denote by $\V$ a given closed symmetric monoidal category.  Since we will avail ourselves of enrichment with respect to each of multiple given categories, we shall include an explicit indication of $\V$ when employing notions such as $\V$-category, $\V$-functor, and so on, omitting the prefix $\V$ only when concerned with the corresponding notions for non-enriched or \textit{ordinary} categories.  Whereas the arrows $f:A \rightarrow A'$ of the \textit{underlying ordinary category} $\A_0$ of a $\V$-category $\A$ are, concretely, arrows $I \rightarrow \A(A,A')$ in $\V$, we nevertheless distinguish notationally between $f$ and the latter arrow in $\V$, which we write as $[f]$ and call \textit{the name of $f$}.  We denote by $\VCAT$ the 2-category of all $\V$-categories.

\begin{ParSub}\label{par:clsmcat_notn}
Given a closed symmetric monoidal category $\V$, we denote by $\uV$ the canonically associated $\V$-category whose underlying ordinary category is isomorphic to $\V$ (and shall be identified with $\V$ whenever convenient); in particular, the internal homs in $\V$ will therefore be denoted by $\uV(V_1,V_2)$.  The canonical `evaluation' morphisms $V_1 \otimes \aeV(V_1,V_2) \rightarrow V_2$ are denoted by $\Ev_{V_1 V_2}$, or simply $\Ev$.
\end{ParSub}

\begin{ParSub}\label{par:cat_classes}
In general, the ordinary categories considered in this paper are not assumed locally small.  Hence whereas we would ideally like to be able to consider \textit{all} ordinary categories as $\SET$-enriched for a single fixed category $\SET$ of \textit{classes} (in an informal sense), we will instead fix a cartesian closed category $\SET$ of sets or classes and tacitly assume that certain given ordinary categories are $\SET$-enriched.  Hence although ideally we would denote by $\CAT$ the category of \textit{all} categories, and by $\MCAT$, $\SMCAT$, and $\SMCCAT$ the 2-categories of all monoidal, symmetric monoidal, and symmetric monoidal closed categories, we will instead employ these names to denote the full sub-2-categories thereof whose objects are $\SET$-enriched categories.  We will denote by $\TWOCAT$ the 2-category of \textit{all} 2-categories.  The reader who objects may instead take $\SET$ to be the usual category of sets and fix a Grothendieck universe $\mathfrak{U}$ with respect to which $\SET$ is $\mathfrak{U}$-small, then take $\CAT$, $\MCAT$, $\SMCAT$ to consist of just the locally small $\mathfrak{U}$-small categories, taking $\VCAT$ for $\V \in \MCAT$ to consist of all $\mathfrak{U}$-small $\V$-categories.  One can then fix a larger universe $\mathfrak{U}'$ containing $\mathfrak{U}$ and take $\TWOCAT$ to instead be the 2-category of $\mathfrak{U}'$-small 2-categories, noting that the 2-categories $\CAT$ and $\VCAT$ will then be objects of $\TWOCAT$.
\end{ParSub}

\section{Pseudomonoids and monoidal $\V$-categories}

Recall that a monoidal category is a category $\M$ equipped with functors \linebreak[4]\mbox{$\otimes:\M \times \M \rightarrow \M$}, $I_\M:\ONE \rightarrow \M$ and natural isomorphisms $a,\ell,r$ such that certain diagrams commute.  Fixing a symmetric monoidal category $\V$ and instead considering a given $\V$-category $\M$ with $\V$-functors $\otimes:\M \otimes \M \rightarrow \M$, $I_\M:\bI \rightarrow \M$ and $\V$-natural isomorphisms $a,\ell,r$ such that the formally analogous diagrams commute, one obtains the notion of \textit{monoidal $\V$-category}, and similarly, the notion of \textit{symmetric monoidal $\V$-category}, equipped with a $\V$-natural \textit{symmetry} $s$.  Here we have employed the \textit{tensor product} $\A \otimes \B$ of $\V$-categories $\A,\B$, and the unit $\V$-category $\bI$, which are part of a symmetric monoidal structure on the category $\VCAT$ of $\V$-categories.

In fact, $\VCAT$ is a \textit{symmetric monoidal 2-category} \cite[\S 1.4]{Ke:Ba}, that is, a symmetric monoidal $\mathbb{CAT}$-category\footnote{Whereas the term \textit{symmetric monoidal 2-category} is interpreted in a more general sense in \cite{GPS,McCr} as denoting a symmetric monoidal bicategory whose underlying bicategory is a 2-category, we employ this term here in the \textit{strict 2-categorical sense}, i.e. the more narrow $\mathbb{CAT}$-enriched sense given above.} where $\mathbb{CAT}$ denotes the cartesian monoidal category of all categories\footnote{The reader who has defined $\VCAT$ to consist of only the $\mathfrak{U}$-small $\V$-categories for a universe $\mathfrak{U}$ \pbref{par:cat_classes} can take $\mathbb{CAT}$ to be the category of $\mathfrak{U}$-small categories, as $\VCAT$ is indeed locally $\mathfrak{U}$-small in this case.}.  Monoidal 2-categories and the more general \textit{monoidal bicategories} have been studied by several authors, e.g. \cite{GPS,DayStr,McCr,SchPr}.  The notion of monoidal $\V$-category can be defined entirely in terms of the structure of $\VCAT$ as a monoidal 2-category, so that a monoidal $\V$-category is equally a \textit{pseudomonoid} in the monoidal 2-category $\VCAT$.  By definition, a pseudomonoid in an arbitrary monoidal 2-category $\K$ consists of an object $\M$ of $\K$ with 1-cells $\bullet:\M \otimes \M \rightarrow \M$, $i:\bI \rightarrow \M$ and invertible 2-cells $a,\ell,r$ subject to certain equations of 2-cells; a definition is given in \cite{McCr}, generalizing slightly the definition given in \cite{DayStr} for $\K$ instead a \textit{Gray monoid}.  The associativity and left-unit 2-cells $a,\ell$ assume the following forms when we omit the associativity 1-cells in $\K$ and write $\lambda$ for the relevant left-unit 1-cell in $\K$,
$$
\xymatrix{
\M\otimes\M\otimes\M \ar[d]_{\bullet\otimes 1} \ar[r]^(.6){1\otimes\bullet} & \M\otimes\M \ar[d]^\bullet & & \bI \otimes \M \ar[d]_{i\otimes 1} \ar@/^3ex/[dr]^{\lambda}|{}="t2" & \\
\M\otimes\M \ar@{}[ur]|(.4){}="s1"|(.6){}="t1"^a \ar[r]_\bullet & \M & & \M\otimes\M \ar@{};"t2"|(.5){}="s3"|(.7){}="t3" \ar[r]_\bullet & \M
\ar@{=>}"s1";"t1"^a \ar@{=>}"s3";"t3"^{\ell}
}
$$
and the right-unit 2-cell $r$ takes an analogous form.  When $\K$ is a \textit{symmetric} monoidal 2-category and so carries a symmetry $\sigma_{\A\B}:\A\otimes\B \rightarrow \B\otimes\A$ $(\A,\B \in \K)$, a pseudomonoid $\M$ in $\K$ is said to be \textit{symmetric} if it is equipped with a 2-cell $s$ from $\M\otimes\M\xrightarrow{\bullet}\M$ to $\M\otimes\M \xrightarrow{\sigma_{\M\M}} \M\otimes\M \xrightarrow{\bullet} \M$ satisfying certain equations \cite{McCr}.

\begin{ExaSub}\label{exa:v_is_a_sm_vcat}
Given a symmetric monoidal closed category $\V$, it follows from \cite[III 6.9, III 7.4]{EiKe} that $\uV$ itself is a symmetric monoidal $\V$-category.  The tensor product $\V$-functor $\otimes:\uV \otimes \uV \rightarrow \uV$ is given on homs by the morphisms 
$$(\uV \otimes \uV)((V,W),(V',W')) = \uV(V,V') \otimes \uV(W,W') \rightarrow \uV(V \otimes W, V' \otimes W')$$
(with $V,W,V',W' \in \V$) obtained as transposes of the composites
$$V \otimes W \otimes \uV(V,V') \otimes \uV(W,W') \xrightarrow{1 \otimes s \otimes 1} V \otimes \uV(V,V') \otimes W \otimes \uV(W,W') \xrightarrow{\Ev \otimes \Ev} V' \otimes W',$$
where $s$ is the symmetry carried by $\V$.
\end{ExaSub}

\begin{ParSub}\label{par:2cats_of_psmons}
The notion of monoidal functor also generalizes in an evident way to that of \textit{monoidal $\V$-functor}\footnote{When applying adjectives like ``monoidal'' in the $\V$-enriched context, we will occasionally include an explicit prefix ``$\V$-'', as in \textit{$\V$-monoidal}, for emphasis or ease of expression.} between monoidal $\V$-categories $\M,\N$, consisting of a $\V$-functor $S:\M \rightarrow \N$ and morphisms $e^S:I_\N \rightarrow SI_\M$ and $m^S_{MM'}:SM \otimes SM' \rightarrow S(M \otimes M')$ $\V$-natural in $M, M' \in \M$, subject to the usual equations.  This notion is an instance of the notion of (lax) \textit{monoidal 1-cell} between pseudomonoids $\M,\N$ in a monoidal 2-category $\K$ (\cite{McCr}, based on a more general notion from \cite{GPS}), consisting of a 1-cell $S:\M \rightarrow \N$ in $\K$ and 2-cells
$$
\xymatrix{
\M \otimes \M \ar[d]_\bullet \ar[r]^{S\otimes S} & \N\otimes \N \ar@{}[dl]|(.4){}="s1"|(.6){}="t1" \ar[d]^\bullet & & & \bI \ar[dl]_i \ar[dr]^i="t2" & & \\
\M \ar[r]_S & \N & & \M \ar@{};"t2"|(.5){}="t3"|(.7){}="s3" \ar[rr]_S & & \N
\ar@{=>}"s1";"t1"_{m^S} \ar@{=>}"s3";"t3"_{e^S}
}
$$
satisfying three equations, which just express in terms of pasting diagrams the familiar componentwise equations for a monoidal functor.  One has in particular the notion of \textit{strong} (resp. \textit{strict}) monoidal 1-cell, where the 2-cells $e^S, m^S$ are required to be isomorphisms (resp. identities).  In the case that $\M$ and $\N$ carry the structure of symmetric pseudomonoids, one may ask that the monoidal 1-cell $S$ also have the property of being \textit{symmetric}, which amounts to just one equation of 2-cells, namely
$$(S \circ s^\M) \cdot m^S = (m^S \circ \sigma_{\M\M}) \cdot (s^\N \circ (S\otimes S))$$
where $\sigma,s^\M,s^\N$ are the symmetries carried by $\K,\M,\N$, respectively, $\circ$ is whiskering, and $\cdot$ is vertical composition.  In particular, this specializes to the notion of \textit{symmetric monoidal $\V$-functor}.  Further, one has the notion of \textit{monoidal 2-cell} between monoidal 1-cells in $\K$; in $\VCAT$, the monoidal 2-cells are \text{monoidal $\V$-natural transformations}.  In particular, one obtains in the case $\V = \mathbb{CAT}$ the notions of \textit{monoidal 2-functor} and \textit{monoidal 2-natural transformation}.

Given a monoidal 2-category $\K$, one thus obtains an associated 2-category \linebreak[4] $\PsMon(\K)$ of pseudomonoids in $\K$ (\cite[\S 2]{McCr}), and for a symmetric monoidal 2-category $\K$, a 2-category $\SymPsMon(\K)$ of symmetric pseudomonoids, with symmetric monoidal 1-cells (\cite[\S 4]{McCr}).  For $\K = \VCAT$, $\PsMon(\K) = \VMCAT$ and $\SymPsMon(\K) = \VSMCAT$ are the 2-categories of monoidal and symmetric monoidal $\V$-categories, respectively.
\end{ParSub}

\begin{PropSub}[{\cite[\S 1,2,4]{McCr}}]\label{thm:mon2func_psmon}\emptybox
Every monoidal 2-functor $G:\K \rightarrow \sL$ lifts to a 2-functor $\PsMon(G):\PsMon(\K) \rightarrow \PsMon(\sL)$.  If $G$ is a symmetric monoidal 2-functor, then $G$ lifts to a 2-functor $\SymPsMon(G):\SymPsMon(\K) \rightarrow \SymPsMon(\sL)$.
\end{PropSub}

\begin{RemSub}\label{rem:desc_psmon_g}
Explicitly, the 2-functor $\PsMon(G)$ of \bref{thm:mon2func_psmon} sends each pseudomonoid $(\M,\bullet,i,a,\ell,r)$ in $\K$ to a pseudomonoid $(G\M,\bullet',i',a',\ell',r')$ in $\sL$, where the 1-cells $\bullet',i'$ and 2-cells $a',\ell',r'$ are obtained from $\bullet,i,a,\ell,r$ by applying $G$ and then composing or whiskering with the following evident canonical 1-cells, determined by the monoidal structure of $G$ (and named uniformly for convenience),
$$
\xymatrix@!C=16ex @R=1ex{
{\mathsf{M}(G,\M):(G\M)^{\otimes 2} \rightarrow G(\M^{\otimes 2})} & & *!<7.6ex,0ex>{\mathsf{E}(G,\M):\bI \rightarrow G\bI}\\
{\mathsf{L}(G,\M):\bI \otimes G\M \rightarrow G(\bI\otimes\M)} & & {\mathsf{R}(G,\M):G\M \otimes \bI \rightarrow G(\M \otimes \bI)}\\
&\mathsf{A}(G,\M):(G\M)^{\otimes 3} \rightarrow G(\M^{\otimes 3})& 
}
$$
so that
$$\bullet' = G(\bullet) \circ \mathsf{M}(G,\M)\;\;\;\;\;\;\;\;\;\;\;i' = G(i) \circ \mathsf{E}(G,\M)$$
$$a' = G(a) \circ \mathsf{A}(G,\M)\;\;\;\;\;\;\;\;\ell' = G(\ell) \circ \mathsf{L}(G,\M)\;\;\;\;\;\;\;\;r' = G(r) \circ \mathsf{R}(G,\M)\;.$$
Similarly, if $\M$ carries a symmetry $s$ then 
$$s' = G(s) \circ \mathsf{M}(G,\M)$$
is the associated symmetry on $G\M$.  The definition of $\PsMon(G)$ on 1-cells also follows this pattern:  Given $S:\M \rightarrow \N$ in $\PsMon(\K)$, the 1-cell $G(S):G\M \rightarrow G\N$ carries monoidal structure 2-cells
$$e^{G(S)} = G(e^S) \circ \mathsf{E}(G,\M)\;\;\;\;\;\;\;\;\;\;\;\;\;\;m^{G(S)} = G(m^S) \circ \mathsf{M}(G,\M)\;.$$
\end{RemSub}

\begin{ParSub}
We shall denote by $\MTWOCAT = \PsMon(\TWOCAT)$ and $\SMTWOCAT = \SymPsMon(\TWOCAT)$ the 2-categories of monoidal 2-categories and symmetric monoidal 2-categories, respectively.
\end{ParSub}

\begin{ThmSub}\label{twofunc_psmon}\emptybox
\begin{enumerate}
\item There are 2-functors
$$\PsMon:\MTWOCAT \longrightarrow \TWOCAT$$
$$\SymPsMon:\SMTWOCAT \longrightarrow \TWOCAT$$
sending each monoidal (resp. symmetric monoidal) 2-category $\K$ to the \mbox{2-category} of pseudomonoids (resp. symmetric pseudomonoids) in $\K$.
\item Given a monoidal 2-natural transformation $\phi:G \Rightarrow H:\K \rightarrow \sL$ and a pseudomonoid $\M$ in $\K$, the 1-cell $\phi_\M:G\M \rightarrow H\M$ is a strict monoidal 1-cell.
\end{enumerate}
\end{ThmSub}
\begin{proof}
For each of the 2-functors to be defined, the needed assignment on 1-cells is provided by \bref{thm:mon2func_psmon}.  We next prove 2, which will furnish the assignment on 2-cells.

It is straightforward to verify 2 directly\footnote{Indeed, one first uses the fact that $\phi$ is monoidal and natural to show that $\phi_\M$ commutes strictly with the product and unit 1-cells carried by $G\M$ and $H\M$, and then one uses the 2-naturality of $\phi$ to verify the needed four equations of 2-cells making $\phi_\M$ strict symmetric monoidal.}, but we can instead give a more conceptual argument, as follows.  The notion of pseudomonoid in $\K$ is an instance of the more general notion of \textit{monoidal homomorphism} of monoidal bicategories (in the terminology of \cite[3.2.2]{SchPr}, called weak monoidal homomorphism in \cite{McCr}).  Indeed, taking $1$ to denote the terminal 2-category, the monoidal homomorphisms $1 \rightarrow \K$ constitute a 2-category $\WMon(1,\K)$ (in the notation of \cite{McCr}), and $\PsMon(\K)$ is precisely the full sub-2-category consisting of those objects whose underlying homomorphism of bicategories is a 2-functor \cite[\S 2]{McCr}.  One has a similar description of $\SymPsMon(\K)$ as a full sub-2-category of the 2-category of \textit{symmetric} (equivalently, \textit{sylleptic}) monoidal homomorphisms $1 \rightarrow \K$ \cite[\S 4]{McCr}.  

Schommer-Pries has shown that symmetric monoidal bicategories constitute a tricategory \cite{SchPr} and in particular that one can whisker the 2-cells therein (called \textit{monoidal transformations}) with the 1-cells \cite[\S 3.3]{SchPr}, namely monoidal homomorphisms.  This whiskering does not depend on the symmetries carried by the 1-cells and so is available in exactly the same form in the non-symmetric case.  Hence, since the given pseudomonoid $\M$ is a monoidal homomorphism $M:1 \rightarrow \K$ and $\phi:G \Rightarrow H:\K \rightarrow \sL$ is a monoidal transformation, we can form the whiskered monoidal transformation $\phi M:GM \Rightarrow HM$, which is therefore a 1-cell of $\WMon(1,\sL)$ and so is a 1-cell of $\PsMon(\sL)$ whose underlying 1-cell in $\sL$ is $\phi_\M:G\M \rightarrow H\M$.  The monoidal structure that $\phi_\M$ thus acquires is given by the two pasting diagrams in \cite[3.3.2]{SchPr} (with $\overline{G} = H$, $F = M$, $\chi^G = m^G$, $\chi^{\overline{G}} = m^H$, $\theta = \phi$), but since $\phi:G \Rightarrow H$ is a monoidal 2-natural transformation of monoidal 2-functors, the 2-cells occupying the resulting pasting diagrams are identity 2-cells.  Hence $\phi_\M$ is strict monoidal as needed, and similarly $\phi_\M$ is strict symmetric monoidal in the symmetric case.

Hence 2 is proved.  Writing $G_\bullet = \PsMon(G)$, the strict monoidal 1-cells $\phi_\M$ of 2 constitute a 2-natural transformation $\phi_\bullet:G_\bullet \Rightarrow H_\bullet$.  Indeed, the naturality of $\phi_\bullet$ is  readily verified using the 2-naturality of $\phi$, and since the 2-functor $\PsMon(\sL) \rightarrow \sL$ is locally faithful, the further 2-naturality of $\phi_\bullet$ follows from that of $\phi$.

This completes the definition of the data for the needed 2-functors.  Using the description of $G_\bullet = \PsMon(G)$ given in \bref{rem:desc_psmon_g}, the equations
$$G_\bullet F_\bullet = (GF)_\bullet\;\;\;\;\;\;\;\;\;\;\;(1_\K)_\bullet = 1_{\PsMon(\K)}$$
for 1-cells $\J \xrightarrow{F} \K \xrightarrow{G} \sL$ in $\MTWOCAT$ follow readily from the fact that
$$\mathsf{T}(GF,\M) = G(\mathsf{T}(F,\M)) \circ \mathsf{T}(G,F\M)\;\;\;\;\;\;\;\;\mathsf{T}(1_\K,\M) = 1$$
for each symbol $\mathsf{T} \in \{\mathsf{M},\mathsf{E},\mathsf{L},\mathsf{R},\mathsf{A}\}$ \pbref{rem:desc_psmon_g}.  The functoriality of $\SymPsMon$ on 1-cells also follows similarly.  Further, $\PsMon$ preserves vertical composition of 2-cells, since for each 2-cell $\phi$ in $\MTWOCAT$ as in 2 above, the components of $\phi_\bullet$ are strict monoidal 1-cells and so are uniquely determined by their underlying 1-cells in $\sL$.  Moreover, $\PsMon$ preserves whiskering for the same reason, since given $\phi$ as in 2 above, together with 1-cells $F:\J \rightarrow \K$ and $T:\sL \rightarrow \sH$, the 2-natural transformations $(\phi F)_\bullet$, $\phi_\bullet F_\bullet$, $(T\phi)_\bullet$, $T_\bullet\phi_\bullet$ also consist of strict monoidal 1-cells.
\end{proof}

\begin{RemSub}
Elements of the above proof suggest that it might be possible to instead obtain the needed 2-functor $\SymPsMon$ as a `subfunctor' of a restriction of a `hom' trihomomorphism $\T(1,-)$ on the tricategory $\T$ of symmetric monoidal bicategories.
\end{RemSub}

\section{Enriched normalizations and comparisons}

Let $\V$ be a symmetric monoidal closed category.

\begin{DefSub}\label{def:can_normn_1}\emptybox
\begin{enumerate}
\item Given a monoidal $\V$-category $\M$, we let
$$U^\M := \M(I,-):\M \rightarrow \uV$$
denote the covariant hom $\V$-functor for the unit object $I$ of $\M$.
\item Given a monoidal $\V$-functor $S:\M \rightarrow \N$, we denote by $\theta^S:U^\M \rightarrow U^{\N}S$ the following composite $\V$-natural transformation:
$$\theta^S := \left(U^{\M} = \M(I,-) \xrightarrow{S} \N(SI,S-) \xrightarrow{\N(e^S,S-)} \N(I,S-) = U^{\N}S\right)\;.$$
\item Given a monoidal $\V$-functor $G:\M \rightarrow \uV$, we define \textit{the comparison morphism} $\kappa^G:U^{\M} \rightarrow G$ as the composite $\V$-natural transformation
$$\kappa^G := \left(U^{\M} \xrightarrow{\theta^G} U^{\uV}G \xrightarrow{\sim} G\right)$$
whose second factor is obtained from the canonical isomorphism $U^{\uV} \cong 1_{\uV}$ by whiskering with $G$.
\end{enumerate}
\end{DefSub}

\begin{PropSub}\label{thm:can_normn_mon}
Let $\M$ be a monoidal $\V$-category.
\begin{enumerate}
\item The $\V$-functor $U^{\M} = \M(I,-):\M \rightarrow \uV$ is $\V$-monoidal, when equipped with $[1_I]:I\rightarrow \M(I,I)$ (the name of the identity arrow on $I$) and the composites
\begin{equation}\label{eq:mon_str_can_normn}\M(I,M) \otimes \M(I,M') \xrightarrow{\otimes} \M(I \otimes I, M \otimes M') \xrightarrow{\M(\ell_I^{-1},1)} \M(I, M \otimes M')\end{equation}
($M,M' \in \M$), noting that $\ell_I = r_I:I \otimes I \rightarrow I$ in $\M$.
\item $U^{\M}$ is moreover symmetric $\V$-monoidal as soon as $\M$ is $\V$-symmetric-monoidal.
\end{enumerate}
\end{PropSub}
\begin{proof}
It is shown in \cite[2.1]{BooSt} that $\M(-,-):\M^\op \otimes \M \rightarrow \uV$ is a monoidal $\V$-functor.  In particular, $\M^\op \otimes \M$ is a monoidal $\V$-category, since moreover $\VMCAT = \PsMon(\VCAT)$ is a symmetric monoidal category, indeed a symmetric monoidal 2-category, as is clear from remarks in the same paper \cite[\S 1, \S 4]{BooSt} (which in fact goes on to consider pseudomonoids in $\VMCAT$, which the authors call \textit{duoidal $\V$-categories}).  Moreover, if $\M$ is symmetric $\V$-monoidal then $\M(-,-)$ is a symmetric monoidal $\V$-functor by \cite[2.3]{BooSt}.  The $\V$-functor $U^\M = \M(I,-)$ can be obtained as a composite
\begin{equation}\label{eq:U_as_composite}\M \xrightarrow{\lambda^{-1}} \text{\bI} \otimes \M \xrightarrow{[I]\otimes 1_\M} \M^\op\otimes\M \xrightarrow{\M(-,-)} \uV\end{equation}
wherein $\bI$ is the unit $\V$-category, $\lambda$ is the left unit isomorphism carried by the monoidal 2-category $\VCAT$, and $[I]:\bI \rightarrow \M^\op$ is the unique $\V$-functor that sends the single object of $\bI$ to $I$.  But $\lambda:\bI\otimes \M \rightarrow \M$ is a strict monoidal $\V$-functor (and indeed serves as the left unit isomorphism carried by the monoidal 2-category $\VMCAT$), and $\lambda$ is symmetric as soon as $\M$ is symmetric; the same conclusions therefore apply to $\lambda^{-1}$.  To equip $[I]:\bI \rightarrow \M^\op$ with the structure of a monoidal $\V$-functor amounts to just endowing the object $I$ with the structure of a monoid in the monoidal category $\M^\op_0$, and indeed $(I,\ell_I^{-1},1_I)$ is a monoid in $\M^\op$ (by the coherence theorem for monoidal categories).  When $\M$ is symmetric, the resulting monoidal $\V$-functor $[I]$ is symmetric since the given monoid in $\M^\op$ is then commutative.  Hence the sequence \eqref{eq:U_as_composite} consists of monoidal (resp. symmetric monoidal) $\V$-functors, and the indicated monoidal structure on $\M(I,-)$ is thus obtained.
\end{proof}

\begin{PropSub}\label{thm:triv_normn}
$U^{\uV} \cong 1_{\uV}$ as monoidal $\V$-functors.
\end{PropSub}
\begin{proof}
Writing $U := U^{\uV}$, we have an isomorphism of $\V$-functors $\xi:1_{\uV} \xrightarrow{\sim} U$ whose components $\xi_V:V \xrightarrow{\sim} \uV(I,V)$ are obtained as the transposes of the canonical isomorphisms $\ell:I \otimes V \rightarrow V$.  It suffices to show that $\xi$ is monoidal.  The unit law for $\xi$ holds since $\xi_I \cdot e^{1_{\uV}} = \xi_I \cdot 1_I = \xi_I = [1_I] = e^{U}:I \rightarrow UI$.  Also, for each pair of objects $V,W \in \V$, the multiplication law
\begin{equation}\label{eq:mult_eqn_for_phi}m^{U}_{VW} \cdot (\xi_V \otimes \xi_W) = \xi_{V\otimes W} \cdot m^{1_{\uV}}_{VW}\;\;:\;\;V \otimes W \rightarrow U(V\otimes W) = \uV(I,V\otimes W)\end{equation}
for $\xi$ is verified by taking the transposes $I\otimes V\otimes W \rightarrow V\otimes W$ of the morphisms on both sides of this equation.  Indeed, using the definition of $m^{U}$ one finds that the transpose of the left-hand side is the clockwise composite around the periphery of the following diagram
$$
\xymatrix@C=12ex @R=4ex{
I\otimes V\otimes W \ar@/_4ex/[ddr]_\ell \ar[r]^{\ell^{-1}\otimes V\otimes W} & I\otimes I\otimes V\otimes W \ar[d]|{1\otimes s\otimes 1} \ar[r]^(.4){I\otimes I\otimes \xi_V\otimes \xi_W} & I\otimes I\otimes \uV(I,V)\otimes\uV(I,W) \ar[d]|{1\otimes s\otimes 1} \ar@/^16ex/[dd]|(.3){I\otimes I\otimes t} \\
 & I\otimes V\otimes I\otimes W \ar[d]|{\ell\otimes\ell} \ar[r]^(.4){I\otimes\xi_V\otimes I\otimes \xi_W} & I\otimes\uV(I,V)\otimes I\otimes\uV(I,W) \ar[dl]|{\Ev\otimes\Ev} \\
 & V \otimes W & I\otimes I\otimes\uV(I\otimes I,V\otimes W) \ar[l]^\Ev
}
$$
wherein the rightmost cell commutes by the definition of the morphism $t := \otimes_{(I,I)(V,W)}$ \pbref{exa:v_is_a_sm_vcat}, the upper square commutes by the naturality of $s$, the triangular cell within the interior commutes by the definition of $\xi$, and the leftmost cell commutes by the coherence theorem for symmetric monoidal categories (\cite[III 1.1]{EiKe}).  But the right-hand side of the needed equation \eqref{eq:mult_eqn_for_phi} is simply $\xi_{V\otimes W} \cdot 1_{V\otimes W} = \xi_{V\otimes W}$, whose transpose is indeed the morphism $\ell$ on the left side of the diagram.
\end{proof}

\begin{PropSub}\label{thm:cmp_uniq_mon_trans}\emptybox
\begin{enumerate}
\item Given a monoidal $\V$-functor $G:\M \rightarrow \uV$, the comparison morphism \linebreak[4]\mbox{$\kappa^G:U^\M \rightarrow G$} is the unique monoidal $\V$-natural transformation $U^\M \rightarrow G$.
\item Given a monoidal $\V$-functor $S:\M \rightarrow \N$, the transformation $\theta^S$ defined in \bref{def:can_normn_1} is the unique monoidal $\V$-natural transformation $U^\M \rightarrow U^\N S$, and $\theta^S = \kappa^{U^\N S}$.
\end{enumerate}
\end{PropSub}
\begin{proof}
Firstly, in the situation of 2, we prove that $\theta^S$ is monoidal.  The unit law for $\theta^S$ is the commutativity of the periphery of the following diagram
$$
\xymatrix{
I \ar[d]_{e^{U^\M}\:=\:[1_I]} \ar[dr]|{[1_{SI}]} \ar[drr]|{[e^S]} \ar[rr]^{e^{U^\N}\:=\:[1_I]} & & *!<-3.4ex,0ex>{\N(I,I) = U^\N I} \ar[d]|{\N(I,e^S)\:=\:U^\N(e^S)}\\
*!<3.4ex,0ex>{U^\M I = \M(I,I)} \ar@/_5ex/[rr]|{\theta^S_I} \ar[r]_S & \N(SI,SI) \ar[r]_(.4){\N(e^S,1)} & *!<-3.4ex,0ex>{\N(I,SI)\:=\:U^\N SI}
}
$$
which clearly commutes.  The multiplication law for $\theta^S$ amounts to the commutativity of the periphery of the following diagram for each pair of objects $L,M \in \M$, since the the top side is $m^{U^\M}_{LM}$, the right side is $\theta^S_{L\otimes M}$, the left side is $\theta^S_L \otimes \theta^S_M$, and the bottom side is $m^{U^\N S}_{LM} = (U^\N m^S_{LM}) \cdot m^{U^\N}_{SL,SM}$.
$$
\renewcommand{\objectstyle}{\scriptstyle}
\renewcommand{\labelstyle}{\scriptstyle}
\xymatrix@C=6ex @R=7ex{
\M(I,L)\otimes\M(I,M) \ar[dd]_{S\otimes S} \ar[rr]^\otimes & & \M(I\otimes I,L\otimes M) \ar[d]_S \ar[r]^{\M(\ell^{-1},1)} & \M(I,L\otimes M) \ar[d]^S \\
& & *!<2.1ex,0ex>{\N(S(I\otimes I),S(L\otimes M))} \ar[r]^(.5){\N(S\ell^{-1},1)} \ar[d]|{\N(m^S,1)} & *!<.6ex,0ex>{\N(SI,S(L\otimes M))} \ar[dd]|{\N(e^S,1)} \\
\N(SI,SL)\otimes\N(SI,SM) \ar[d]|{\N(e^S,1)\otimes\N(e^S,1)} \ar[r]^\otimes & \N(SI\otimes SI,SL\otimes SM) \ar[d]|{\N(e^S\otimes e^S,1)} \ar[r]^{\N(1,m^S)} & \N(SI\otimes SI,S(L\otimes M)) \ar[d]|{\N((e^S \otimes e^S)\cdot \ell_I^{-1},1)} & \\
\N(I,SL)\otimes\N(I,SM) \ar[r]_\otimes & \N(I\otimes I,SL\otimes SM) \ar[r]_{\N(\ell^{-1},m^S)} & \N(I,S(L\otimes M)) \ar@{=}[r] & \N(I,S(L\otimes M))
}
$$
The large cell at the top-left corner commutes by the $\V$-naturality of $m^S$, the square at the bottom-left commutes by the $\V$-functoriality of $\otimes$, the square immediately to the right of this clearly commutes, and the square at the top-right corner commutes by the $\V$-functoriality of $S$.  The remaining cell at the bottom-right commutes since it follows from the fact that $S$ is a monoidal functor that $S\ell^{-1}_I \cdot e^S = m^S_{II} \cdot (e^S \otimes e^S) \cdot \ell^{-1}_I$.

Having thus shown that $\theta^S$ is monoidal, it now follows (in view of \ref{thm:triv_normn}) that $\kappa^G$ is monoidal in the situation of $1$.  For the needed uniqueness of $\kappa^G$ (and hence also of $\theta^S$), observe that given any $\phi:U^\M \rightarrow G$ in $\VMCAT$, we have for each $M \in \M$ a diagram
$$
\xymatrix@C=12ex @R=6ex{
\M(I,M) \ar@/_18ex/[ddd]_1 \ar[d]^{\M(I,-)_{IM}} \ar[r]^{G_{IM}} & \uV(GI,GM) \ar[d]_{\uV(\phi_I,1)} \ar@/^14ex/[dd]^{\uV(e^G,1)} \\
\uV(\M(I,I),\M(I,M)) \ar[d]^{\uV([1_I],1)} \ar[r]^{\uV{(1,\phi_M)}} & \uV(\M(I,I),GM) \ar[d]_{\uV([1_I],1)} \\
\uV(I,\M(I,M)) \ar[d]^\wr \ar[r]^{\uV(1,\phi_M)} & \uV(I,GM) \ar[d]^\wr \\
\M(I,M) \ar[r]_{\phi_M} & GM
}
$$
in $\V$.  The upper rectangle commutes by the $\V$-naturality of $\phi$, the rightmost cell commutes by the monoidality of $\phi$, and the other cells clearly commute.  But the clockwise composite around the periphery is $\kappa^G_M$, so  $\phi_M = \kappa^G_M$.
\end{proof}

\begin{DefSub}\label{def:normal_normalization}\emptybox
\begin{enumerate}
\item We say that a monoidal $\V$-functor $G:\M \rightarrow \N$ is \textit{$\V$-normal} if $\theta^G:U^{\M} \rightarrow U^{\N}G$ is an isomorphism.
\item A \textit{$\V$-normalization} for a given monoidal $\V$-category $\M$ is a monoidal $\V$-functor $G:\M \rightarrow \uV$ that is $\V$-normal, equivalently, whose comparison $\kappa^G:U^{\M} \rightarrow G$ is an isomorphism.
\item We call $U^\M:\M \rightarrow \uV$ the \textit{canonical $\V$-normalization} for $\M$.
\end{enumerate}
\end{DefSub}

\begin{RemSub}
In the case that $\M$ is a symmetric monoidal $\V$-category, any $\V$-normalization $G:\M \rightarrow \uV$ is necessarily symmetric, since $G \cong U^{\M}$ as monoidal $\V$-functors and $U^{\M}$ is symmetric \pbref{thm:can_normn_mon}.
\end{RemSub}

\begin{RemSub}\label{rem:ord_norml_normn}
In the case of $\V = \SET$, we recover from \bref{def:normal_normalization} 1 the notion of \textit{normal monoidal functor} employed in \cite[\S 2]{Ke:Doctr}.  The notion of $\V$-normalization for $\V = \SET$ coincides (up to a bijection) with the notion of \textit{normalization} of a monoidal category given in \cite{EiKe}, which there is defined as an ordinary functor $G:\M \rightarrow \SET$ equipped with a \textit{specified} isomorphism $\M(I,-) \cong G$ (which, upon transporting the monoidal structure on the latter functor to $G$, coincides with $\theta^G$).
\end{RemSub}

\begin{PropSub}\label{thm:vnormality_inv_iso}
If $G,H:\M \rightarrow \N$ are isomorphic monoidal $\V$-functors and $G$ is $\V$-normal, then $H$ is $\V$-normal.
\end{PropSub}
\begin{proof}
This follows from the uniqueness of $\theta^H:U^\M \Rightarrow U^\N H$ \pbref{thm:cmp_uniq_mon_trans}.
\end{proof}

\section{Change of base for symmetric monoidal $\V$-categories}

\begin{ParSub} \label{par:ch_base}
Recall that there is a 2-functor $\eCAT{(-)}:\MCAT \rightarrow \TWOCAT$ (\cite{EiKe,Cr}) from the 2-category of monoidal categories to the 2-category of 2-categories, sending each monoidal functor $G:\V \rightarrow \W$ to the \textit{change-of-base} 2-functor $G_\ast:\VCAT \rightarrow \WCAT$.  For a $\V$-category $\A$, the $\W$-category $G_\ast\A$ has objects those of $\A$ and homs given by
$$(G_\ast\A)(A,A') = G\A(A,A')\;\;\;\;(A,A' \in \A)\;.$$
Given a $\V$-functor $P:\A \rightarrow \B$, the $\W$-functor $G_\ast(P):G_\ast\A \rightarrow G_\ast\B$ is given on objects as $P$ and homs as $G(P_{AA'}):G\A(A,A') \rightarrow G\B(PA,PA')$.  Given a 2-cell $\psi:P \Rightarrow Q:\A \rightarrow \B$ in $\VCAT$, the components of $G_\ast\psi:G_\ast(P) \Rightarrow G_\ast(Q)$ are the arrows $(G_\ast\psi)_A:PA \rightarrow QA$ in $G_\ast\B$ with names
$$I_\W \xrightarrow{e^G} GI_\V \xrightarrow{G[\psi_A]} G\B(PA,QA)\;\;\;\;\;(A \in \A)\;.$$
Given a monoidal transformation $\phi:G \Rightarrow H$, where $G,H:\V \rightarrow \W$ are monoidal functors, we obtain an associated 2-natural transformation $\phi_\ast:G_\ast \Rightarrow H_\ast$ whose components are identity-on-objects $\W$-functors
$$\phi_\ast\A:G_\ast\A \rightarrow H_\ast\A\;,\;\;\;\;\A \in \VCAT$$
whose structure morphisms are simply the components
$$(\phi_\ast\A)_{A A'} = \phi_{\A(A,A')}:G\A(A,A') \rightarrow H\A(A,A')\;\;\;\;\;\;(A,A' \in \A)$$
of $\phi$.
\end{ParSub}

\begin{ExaSub}\label{exa:underlying_ordinary}
Given a monoidal category $\V$, the canonical normalization $U^\V = \V(I,-):\V \rightarrow \SET$ determines a change-of-base 2-functor $(-)_0 = U^\V_\ast:\VCAT \rightarrow \CAT$ sending each $\V$-category $\A$ to its underlying ordinary category $\A_0$.
\end{ExaSub}

Cruttwell \cite{Cr} has proved that the change-of-base 2-functor associated to a braided monoidal functor is a monoidal 2-functor\footnote{Indeed, the proof of \cite[5.7.1]{Cr} shows (as stated in the introduction to Ch. 5 there) that the change-of-base 2-functor $N_*$ is (not only a weak monoidal homomorphism but in fact) a monoidal 2-functor.}; further, we have the following.

\begin{PropSub}\label{ch_base_val_in_sm2cat}\emptybox
\begin{enumerate}
\item Given a symmetric monoidal functor $G:\V \rightarrow \W$, the change-of-base 2-functor $G_\ast:\VCAT \rightarrow \WCAT$ carries the structure of a symmetric monoidal 2-functor.
\item The 2-functor $\eCAT{(-)}:\MCAT \rightarrow \TWOCAT$ lifts to a 2-functor
$$\eCAT{(-)}:\SMCAT \rightarrow \SMTWOCAT,\;\;\V \mapsto \VCAT$$
valued in the 2-category $\SMTWOCAT$ of symmetric monoidal 2-categories.
\end{enumerate}
\end{PropSub}
\begin{proof}
1.  By the preceding comment, $G_\ast$ is a monoidal 2-functor.  Explicitly, for each pair $\A,\B$ of $\V$-categories, there is an identity-on-objects $\W$-functor
$$G_\ast\A \otimes G_\ast\B \rightarrow G_\ast(\A \otimes \B)$$
given on homs as just the components
$$m^G:G\A(A,A') \otimes G\B(B,B') \rightarrow G(\A(A,A')\otimes\B(B,B'))$$
($A,A' \in \A$, $B,B' \in \B$) of the structure morphism $m^G$ carried by $G$.  This 2-natural family of $\W$-functors, together with the identity-on-objects $\W$-functor $\bI_\W \rightarrow G_\ast\bI_\V$ determined by $e^G:I_\W \rightarrow GI_\V$, constitute the needed structure on $G_\ast$.  The symmetry of $G_\ast$ follows immediately from that of $G$.

2.  We know that $\eCAT{(-)}:\SMCAT \rightarrow \TWOCAT$ is 2-functorial, and we immediately verify that the assignment of monoidal structures on $1$-cells given by 1 respects composition, so that we obtain a functor $\eCAT{(-)}:\SMCAT \rightarrow \SMTWOCAT$.  This is moreover a 2-functor, since given a 2-cell $\psi:G \Rightarrow H:\V \rightarrow \W$ in $\SMCAT$, it follows immediately from the monoidality of $\psi$ that $\psi_\ast:G_\ast \Rightarrow H_\ast$ is monoidal.
\end{proof}

\begin{ThmSub}\label{thm:ch_base_vsmcats}\emptybox
\begin{enumerate}
\item There is a 2-functor
$$\eSMCAT{(-)}:\SMCAT \rightarrow \TWOCAT,\;\;\V \mapsto \VSMCAT$$
sending each symmetric monoidal category $\V$ to the 2-category $\VSMCAT$ of \linebreak[4]\mbox{symmetric monoidal} $\V$-categories.
\item Given a 2-cell $\phi:G \Rightarrow H:\V \rightarrow \W$ in $\SMCAT$ and a symmetric monoidal $\V$-category $\M$, the identity-on-objects $\W$-functor \mbox{$\phi_\ast\M:G_\ast\M \rightarrow H_\ast\M$} is strict symmetric $\W$-monoidal.
\end{enumerate}
\end{ThmSub}
\begin{proof}
This follows from \bref{ch_base_val_in_sm2cat} and \bref{twofunc_psmon}, since the needed 2-functor is obtained as the composite
$$\SMCAT \xrightarrow{\eCAT{(-)}} \SMTWOCAT \xrightarrow{\SymPsMon} \TWOCAT\;.$$
\end{proof}

\begin{RemSub}\label{par:ch_base_2func_on_smvcats}
By \bref{thm:ch_base_vsmcats}, each 1-cell $G:\V \rightarrow \W$ in $\SMCAT$ determines a 2-functor
$$G_*:\VSMCAT \rightarrow \eSMCAT{\W}$$
which we again write as $G_*$ by abuse of notation.  Recall that there is a unique 2-cell $\theta^G:U^\V \Rightarrow U^\W G:\V \rightarrow \SET$ \pbref{thm:cmp_uniq_mon_trans} in $\SMCAT$, where $U^\V:\V \rightarrow \SET$ and $U^\W:\W \rightarrow \SET$ are the canonical normalizations.  By \bref{thm:ch_base_vsmcats}, each symmetric monoidal $\V$-category $\M$ therefore determines an identity-on-objects strict symmetric monoidal functor
\begin{equation}\label{eq:can_comp}\theta^G_*\M:\M_0 = U^\V_*\M \rightarrow U^\W_*G_*\M = (G_*\M)_0\end{equation}
which serves as a canonical comparison of the underlying ordinary symmetric monoidal categories.  We shall later make use of the observation that, given a 1-cell $S:\M \rightarrow \N$ in $\VSMCAT$, the monoidal structure morphisms carried by the symmetric monoidal $\W$-functor $G_*(S):G_*\M \rightarrow G_*\N$ are obtained from those of $S$ by applying $\theta^G_*\N:\N_0 \rightarrow (G_*\N)_0$.  Indeed, this follows from the fact that $\theta^G_*:U^\V_* \Rightarrow U^\W_*G_*:\VSMCAT \rightarrow \SMCAT$ is natural and consists of strict monoidal functors.  Also, given a 2-cell $\phi:S \Rightarrow T:\M \rightarrow \N$ in $\VSMCAT$, the components of $G_*(\phi):G_*(S) \Rightarrow G_*(T)$ are obtained from those of $\phi$ by applying that same comparison functor $\theta^G_*\N$.
\end{RemSub}

\section{Symmetric monoidal closed $\V$-categories and change of base}

\begin{DefSub}\label{def:vsmccat}
Let $\V$ be symmetric monoidal closed category.  A \textit{symmetric monoidal closed $\V$-category} is a symmetric monoidal $\V$-category $\M$ such that for each object $M \in \M$, the $\V$-functor $M \otimes (-):\M \rightarrow \M$ has a right $\V$-adjoint, denoted by $\uM(M,-)$.  We denote by $\VSMCCAT$ the full sub-2-category of $\VSMCAT$ with objects all symmetric monoidal closed $\V$-categories.
\end{DefSub}

\begin{ExaSub}\label{v_is_a_vsmc_vcat}
Given a symmetric monoidal closed category $\V$, we obtain a symmetric monoidal closed $\V$-category $\uV$ by \bref{exa:v_is_a_sm_vcat}, since the adjunctions $V \otimes (-) \dashv \uV(V,-)$ associated to the objects $V$ of $\V$ are $\V$-adjunctions.
\end{ExaSub}

\begin{ThmSub}\label{thm:ch_base_vsmccats}
Let $G:\V \rightarrow \W$ be a 1-cell in $\SMCCAT$.
\begin{enumerate}
\item Given a symmetric monoidal closed $\V$-category $\M$, the associated symmetric monoidal $\W$-category $G_*\M$ is $\W$-symmetric-monoidal-closed.
\item The change-of-base 2-functor $G_*:\VSMCAT \rightarrow \WSMCAT$ lifts to a 2-functor
$$G_*:\VSMCCAT \rightarrow \WSMCCAT\;,$$
which we denote also by $G_*$, by abuse of notation.
\item There is a 2-functor
$$\eSMCCAT{(-)}:\SMCCAT \rightarrow \TWOCAT,\;\;\;\;\V \mapsto \VSMCCAT\;.$$
\end{enumerate}
\end{ThmSub}
\begin{proof}
It suffices to prove 1, from which 2 and 3 follow.  Denote by $\otimes$ and $\boxtimes$ the tensor products carried by $\M$ and $G_*\M$, respectively.  For each $M \in \M$, we have a $\V$-adjunction $M \otimes (-) \dashv \uM(M,-)$, and, applying the 2-functor $G_*:\VCAT \rightarrow \WCAT$, we obtain a $\W$-adjunction.  Hence, it suffices to show that $G_*(M \otimes (-)) = M \boxtimes (-):G_*\M \rightarrow G_*\M$, as $\W$-functors.  The tensor product $\boxtimes$ is given on object just as $\otimes$, so on objects, the needed equation is immediate.  But on homs, $\boxtimes$ is given by the composites
$$
\renewcommand{\objectstyle}{\scriptstyle}
\renewcommand{\labelstyle}{\scriptstyle}
\xymatrix@C=7ex @R=1ex{
G\M(M,M') \otimes G\M(N,N') \ar[r]^{m^G} & G(\M(M,M') \otimes \M(N,N')) \ar[r]^{G(\otimes)} & G\M(M \otimes N, M' \otimes N')
}
$$
for all $M,M',N,N' \in \M$.  Therefore, the morphism $(M \boxtimes (-))_{N,N'}:G\M(N,N') \rightarrow G\M(M \otimes N, M \otimes N')$ associated with each pair $N, N' \in \M$ is, by definition, obtained as the counter-clockwise composite around the periphery of the following diagram
$$
\xymatrix{
G\M(N,N') \ar[dd]_{\ell^{-1}} \ar[dr]|{G\ell^{-1}} \ar[rr]^{G((M \otimes (-))_{N N'})} & & G\M(M \otimes N, M \otimes N')\\
& G(I \otimes \M(N,N')) \ar[r]^(.43){G([1] \otimes 1)} & G(\M(M,M) \otimes \M(N,N')) \ar[u]_{G(\otimes_{(M,N)(M,N')})}\\
I \otimes G\M(N,N') \ar[r]_{e^G \otimes 1} & GI \otimes G\M(N,N') \ar[u]_{m^G} \ar[r]_(.4){G[1] \otimes 1} & G\M(M,M) \otimes G\M(N,N') \ar[u]_{m^G}
}
$$
which commutes, using the definition of $M \otimes (-)$, the naturality of $m^G$, and the fact that $G$ is a monoidal functor.
\end{proof}

\begin{PropSub}\label{thm:can_normns_commute_strictly}
Let $\M$ be a monoidal $\V$-category.  Then the diagram of symmetric monoidal functors
$$
\xymatrix{
\M_0 \ar[dr]_{U^{\M_0}} \ar[rr]^{U^\M_0} &      & \V \ar[dl]^{U^\V}\\
                                       & \SET & 
}
$$
commutes (strictly), where $U^\M$ is the canonical $\V$-normalization of $\M$, and $U^{\M_0}$, $U^\V$ are the canonical normalizations of $\M_0$ and $\V$, respectively.  Equivalently, the unique monoidal transformation $\theta^{U^\M}:U^{\M_0} \rightarrow U^\V U^\M_0$ \pbref{thm:cmp_uniq_mon_trans} is an identity morphism.
\end{PropSub}
\begin{proof}
Denote the unit objects of $\V$ and $\M$ by $I$ and $J$, respectively.  Then for each object $M \in \M$,
$$U^\V U^\M M = \V(I,\M(J,M)) = \M_0(J,M) = U^{\M_0}M\;,$$
so the diagram commutes on objects.  It suffices to show that $\theta^{U^\M}_M:U^{\M_0}M \rightarrow U^\V U^\M M$ is the identity map.  Since by definition $e^{U^\M}_M = [1_J]:I \rightarrow U^\M J = \M(J,J)$, we find that $\theta^{U^\M}_M$ is, by definition, the composite
$$\M_0(J,M) \xrightarrow{U^\M_{JM}} \V(\M(J,J),\M(J,M)) \xrightarrow{\V([1_J],1)}\V(I,\M(J,M)) = \M_0(J,M)\;.$$
Since $U^\M = \M(J,-)$, this composite sends each morphism $f:J \rightarrow M$ in $\M_0$ to the composite
$$I \xrightarrow{[1_J]} \M(J,J) \xrightarrow{\M(J,f)} \M(J,M)\;,$$
which, viewed as morphism $J \rightarrow M$ in $\M_0$, is exactly $f$.
\end{proof}

The following theorem shows that any symmetric monoidal closed $\V$-category $\M$ can be recovered (up to isomorphism) from its underlying symmetric monoidal closed category $\M_0$ together with the monoidal functor $U^\M_0 = \M(I,-):\M_0 \rightarrow \V$, which we shall also write as simply $U^\M$, by abuse of notation.  We denote by $\uM$ the $\M_0$-enriched category $\underline{\M_0}$.

\begin{ThmSub}\label{thm:rec_vsmcc_from_ordinary_data}
Let $\M$ be a symmetric monoidal closed $\V$-category.  Then there is an isomorphism of symmetric monoidal closed $\V$-categories
$$U^\M_*\uM \cong \M\;,$$
and this isomorphism is identity-on-objects and strict symmetric $\V$-monoidal.  \end{ThmSub}

In order to prove this theorem, we will make use of the following notion, closely related to the notion of \textit{$\V$-profunctor monad}.

\begin{DefSub}\label{def:superp}
Let $\A$ be a $\V$-category.  A \textit{$\V$-category superposed upon $\A$} consists of a $\V$-functor $\B(-,-):\A^\op\otimes\A \rightarrow \uV$ and families of morphisms
$$\circ_{ABC}:\B(A,B)\otimes\B(B,C) \rightarrow \B(A,C)\;\;\;\;\;\;\;\;j_A:I \rightarrow \B(A,A)$$
extraordinarily $\V$-natural in $A,B,C \in \A$, resp. $A \in \A$, such that
$$\B = (\ob\A,(\B(A,B))_{A,B\in\ob\A},\circ,j)$$
is a $\V$-category.  We then denote by
$$\A(A,B)\otimes\B(B,C) \xrightarrow{\bullet_{A(BC)}} \B(A,C)\;\;\;\;\;\;\B(A,B)\otimes\A(B,C) \xrightarrow{\bullet_{(AB)C}} \B(A,C)$$
the transposes of $\B(-,C)_{AB}$ and $\B(A,-)_{BC}$, respectively.
\end{DefSub}

\begin{PropSub}\label{thm:vfunc_ind_by_superposed_vcat}
Let $\B$ be a $\V$-category superposed upon $\A$.  Then there is an identity-on-objects $\V$-functor $S:\A \rightarrow \B$ given by the following composite morphisms
$$\A(A,B) \xrightarrow{\ell^{-1}} I\otimes\A(A,B)\xrightarrow{j_A\otimes 1} \B(A,A)\otimes\A(A,B) \xrightarrow{\bullet_{(AA)B}} \B(A,B)$$
$(A,B \in \A)$, which by the extraordinary $\V$-naturality of $j$ are equal to
$$\A(A,B) \xrightarrow{r^{-1}} \A(A,B)\otimes I \xrightarrow{1\otimes j_B}\A(A,B)\otimes\B(B,B) \xrightarrow{\bullet_{A(BB)}} \B(A,B)\;.$$
\end{PropSub}
\begin{proof}
For each object $B \in \A$, the morphism $j_B:I \rightarrow \B(B,B)$ can be construed as an `element' of the $\V$-presheaf $\B(-,B):\A^\op \rightarrow \uV$ and so, by the Yoneda lemma, determines a $\V$-natural transformation $S_{(-)B}:\A(-,B) \Rightarrow \B(-,B)$.  We claim that the resulting morphisms $S_{AB}:\A(A,B) \rightarrow \B(A,B)$ constitute an identity-on-objects $\V$-functor $S:\A \rightarrow \B$.  By its definition, $S_{AA}$ commutes with the identity morphisms $I \rightarrow \A(A,A)$ and $j_A:I \rightarrow \B(A,A)$.  Given objects $A,B,C \in \A$, it suffices to show that the following diagram commutes.
$$
\xymatrix{
\A(A,B)\otimes\A(B,C) \ar[d]_{c_{ABC}} \ar[r]^{1\otimes S_{BC}} & \A(A,B)\otimes\B(B,C) \ar[d]_{\bullet_{A(BC)}} \ar[r]^{S_{AB}\otimes 1} & \B(A,B)\otimes\B(B,C) \ar[d]^{\circ_{ABC}} \\
\A(A,C) \ar[r]_{S_{AC}} & \B(A,C) \ar@{=}[r] & \B(A,C)
}
$$
Since the composition morphism $c_{ABC}$ can be described equally as the transpose  of $\A(-,C)_{AB}$, and $\bullet_{A(BC)}$ is by definition the transpose of $\B(-,C)_{AB}$, the leftmost square commutes by the $\V$-naturality of $S_{(-)C}$.  Further, it is straightforward to verify that the $\V$-naturality of $\circ_{(-)BC}$ entails that the rightmost square commutes.
\end{proof}

\begin{LemSub}\label{thm:mon_superp}
Let $\B$ be a $\V$-category superposed upon a symmetric monoidal $\V$-category $\A$, and suppose that we are given morphisms
$$\B(A,A')\otimes\B(B,B') \xrightarrow{\boxtimes_{(A,B)(A',B')}} \B(A\otimes B,A'\otimes B')$$
$\V$-natural in $(A,B),(A',B') \in \A\otimes\A$ that yield a $\V$-functor $\boxtimes:\B\otimes\B \rightarrow \B$ given on objects as the tensor product $\otimes$ carried by $\A$.  Then the diagram of $\V$-functors
\begin{equation}\label{eq:mon_superp}
\xymatrix@C=4ex @R=3ex{
\A\otimes\A \ar[d]_\otimes \ar[r]^{S\otimes S} & \B\otimes\B \ar[d]^\boxtimes\\
\A \ar[r]_S & \B
}
\end{equation}
commutes, where $S$ is as defined in \bref{thm:vfunc_ind_by_superposed_vcat}
\end{LemSub}
\begin{proof}
The diagram clearly commutes on objects, and the two composite $\V$-functors are given on homs by the composites on the periphery of the following diagram
$$
\renewcommand{\objectstyle}{\scriptstyle}
\renewcommand{\labelstyle}{\scriptstyle}
\xymatrix@C=8ex @R=2.5ex{
\A(A,A')\otimes\A(B,B') \ar[dd]_{t_{(A,B)(A',B')}} \ar[r] & \B(A,A)\otimes\A(A,A')\otimes\B(B,B)\otimes\A(B,B') \ar[r]^(.6){\bullet \otimes \bullet} \ar[d]^{1\otimes s\otimes 1} & \B(A,A')\otimes\B(B,B') \ar[dd]^{\boxtimes_{(A,B)(A',B')}}\\
 & \B(A,A)\otimes\B(B,B)\otimes\A(A,A')\otimes\A(B,B') \ar[d]^{\boxtimes_{(A,B)(A,B)} \otimes t_{(A,B)(A',B')}} & \\
\A(A\otimes B,A'\otimes B') \ar[r]_(.4){(j_{A\otimes B}\otimes 1) \circ \ell^{-1}} & \B(A\otimes B,A\otimes B)\otimes \A(A\otimes B,A'\otimes B') \ar[r]_(.6){\bullet} & \B(A\otimes B,A'\otimes B')
}
$$
where we have written $t$ for the structure morphisms of the tensor product $\V$-functor carried by $\A$, and the unlabelled arrow is $((j_A\otimes 1) \cdot \ell^{-1})\otimes((j_B\otimes 1) \cdot \ell^{-1})$.  But the leftmost cell commutes since the $\V$-functor $\boxtimes$ preserves identity arrows, and the rightmost cell commutes by the $\V$-naturality of the given morphisms $\boxtimes$.
\end{proof}

\begin{proof}[\textnormal{\textbf{Proof of \boldref{thm:rec_vsmcc_from_ordinary_data}.}}]
Writing $U = U^\M$, we shall equip $U_*\uM$ with the structure of a $\V$-category superposed upon $\M$ \pbref{def:superp}, as follows.  We have a $\V$-functor $\uM(-,-):\M^\op\otimes\M \rightarrow \M$ (by \cite[\S 1.10]{Ke:Ba}), and so we obtain a composite $\V$-functor
\begin{equation}\label{eq:superp_profunctor}\M^\op\otimes\M \xrightarrow{\uM(-,-)} \M \xrightarrow{U} \uV\;,\end{equation}
with respect to which the composition and unit morphisms
$$U\uM(L,M)\otimes U\uM(M,N) \xrightarrow{m^{U}} U(\uM(L,M)\otimes\uM(M,N)) \xrightarrow{U(c_{LMN})} U\uM(L,N)$$
\begin{equation}\label{eq:um_unit}I \xrightarrow{e^{U}} U I \xrightarrow{U[1_M]} U\uM(M,M)\end{equation}
for $U_*\uM$ are extraordinarily $\V$-natural\footnote{The composition morphisms $c_{LMN}:\uM(L,M)\otimes\uM(M,N) \rightarrow \uM(L,N)$ are extraordinarily $\V$-natural in $L,M,N \in \M$ since they are defined as mates of the composites $L\otimes\uM(L,M)\otimes\uM(M,N) \xrightarrow{\Ev_{LM}\otimes 1} M\otimes\uM(M,N) \xrightarrow{\Ev_{MN}} N$.  A similar remark applies to the morphisms $[1_M]:I \rightarrow \uM(M,M)$.} in $L,M,N \in \M$, resp. $M \in \M$.

By \bref{thm:vfunc_ind_by_superposed_vcat} we obtain an identity-on-objects $\V$-functor $S:\M \rightarrow U_*\uM$, and we claim that the structure morphisms $S_{MN}$ for $S$ are equal to the canonical isomorphisms
\begin{equation}\label{eq:can_isos_k}\M(M,N) \xrightarrow{\M(r,N)} \M(M\otimes I,N) \xrightarrow{\sim} \M(I,\uM(M,N)) = U\uM(M,N)\;.\end{equation}
Indeed, in the proof of \bref{thm:vfunc_ind_by_superposed_vcat} we defined $S_{(-)N}:\M(-,N) \Rightarrow U\uM(-,N)$ as the $\V$-natural transformation induced by the unit $I \rightarrow U\uM(N,N)$ \eqref{eq:um_unit}, but one readily finds that the canonical $\V$-natural isomorphism $\M(-,N) \Rightarrow U\uM(-,N)$ of \eqref{eq:can_isos_k} corresponds via the Yoneda lemma to this same unit element.  Therefore $S$ is an isomorphism of $\V$-categories, so it suffices to show that $S$ is a strict symmetric monoidal $\V$-functor.  The tensor product $\V$-functor $\boxtimes$ carried by $U_*\uM$ is given on homs by the composites
$$
\renewcommand{\objectstyle}{\scriptstyle}
\renewcommand{\labelstyle}{\scriptstyle}
\xymatrix@C=8ex @R=2.5ex{
U\uM(M,M')\otimes U\uM(N,N') \ar[r]^{m^{U}} & U(\uM(M,M')\otimes\uM(N,N')) \ar[rr]^{U(\otimes_{(M,N)(M',N')})} & & U\uM(M\otimes N,M'\otimes N')
}
$$
which are $\V$-natural in $(M,N),(M',N') \in \M\otimes\M$ with respect to the $\V$-functor \eqref{eq:superp_profunctor}, so by \bref{thm:mon_superp} the diagram \eqref{eq:mon_superp} commutes with $\A := \M$ and $\B := U_*\uM$.  Hence it suffices to show that $S$ strictly preserves the canonical isomorphisms $a,\ell,r,s$.  By \bref{thm:can_normns_commute_strictly}, the underlying ordinary symmetric monoidal categories $(U_*\uM)_0$ and $(\uM)_0$ are identical, so since $S$ is given on homs by the canonical isomorphisms \eqref{eq:can_isos_k} it is immediate that $S_0:\M_0 \rightarrow (\uM)_0$ is the canonical isomorphism $\M_0 \cong (\underline{\M_0})_0$ \pbref{par:clsmcat_notn}, which is strict symmetric monoidal.
\end{proof}

\begin{RemSub}\label{rem:can_isos_for_vsmccat}
In the situation of \bref{thm:rec_vsmcc_from_ordinary_data}, we deduce by \bref{thm:can_normns_commute_strictly} that the strict monoidal comparison functor $\theta^{U^\M}_*:(\uM)_0 \rightarrow (U^\M_*\uM)_0$ \pbref{eq:can_comp} is an identity functor.  Further, as noted within the above proof of \bref{thm:rec_vsmcc_from_ordinary_data}, the isomorphism of ordinary symmetric monoidal categories that underlies the isomorphism $U^\M_*\uM \cong \M$ of \bref{thm:rec_vsmcc_from_ordinary_data} is exactly the familiar canonical isomorphism $(\uM)_0 \cong \M_0$ \pbref{par:clsmcat_notn}.
\end{RemSub}

\section{Enrichment of a symmetric monoidal closed functor}\label{sec:enr_smcfunc}

\begin{DefSub}[Eilenberg-Kelly \cite{EiKe}]\label{def:enr_smc_func}
Given a symmetric monoidal functor $G:\V \rightarrow \W$ between closed symmetric monoidal categories, we obtain an associated $\W$-functor $\grave{G}:G_*\uV \rightarrow \uW$, given on objects just as $G$, with each morphism
\begin{equation}\label{eq:str_morph_mgrave}\grave{G}_{V_1 V_2} : (G_*\uV)(V_1,V_2) = G\uV(V_1,V_2) \rightarrow \uW(GV_1,GV_2)\;\;\;\;(V_1,V_2 \in \V)\end{equation}
obtained as the transpose of the composite
$$GV_1 \otimes G\uV(V_1,V_2) \xrightarrow{m^G} G(V_1 \otimes \uV(V_1,V_2)) \xrightarrow{G(\Ev)} GV_2\;.$$
Indeed, by \cite{EiKe}, $G$ determines a \textit{closed functor} $\V \rightarrow \W$, whose defining data include the structure morphisms \eqref{eq:str_morph_mgrave}, and the associated $\W$-functor $\grave{G}$ is obtained via \cite[I 6.6]{EiKe}.
\end{DefSub}

The following lemma is fundamental to several of the theorems in the sequel.

\begin{LemSub}[Fundamental Lemma]\label{thm:fund_lemma}
Let $G:\M \rightarrow \N$ be a 1-cell in $\VSMCCAT$.  Neglecting to distinguish notationally between 1-cells in $\VSMCCAT$ and their `underlying' 1-cells in $\SMCCAT$ \textnormal{(}\bref{exa:underlying_ordinary},\bref{thm:ch_base_vsmccats}\textnormal{)}, we have a commutative diagram of $\V$-functors
\begin{equation}\label{eq:fund_lem_diag}
\xymatrix{
U^\M_*\uM \ar[d]^\wr \ar[rr]^{\theta^G_*\uM} & & U^\N_*G_*\uM \ar[rr]^{U^\N_*(\grave{G})} & & U^\N_*\uN \ar[d]^\wr \\
\M \ar[rrrr]_G & & & & \N
}
\end{equation}
where $U^\M:\M \rightarrow \V$ and $U^\N:\N \rightarrow \V$ are the canonical $\V$-normalizations of $\M$ and $\N$, respectively, and the vertical arrows are the canonical isomorphisms \pbref{thm:rec_vsmcc_from_ordinary_data}.
\end{LemSub}
\begin{proof}
Since $U^\N_*(\grave{G})$ and $G$ are given in the same way on objects, and the other $\V$-functors in the diagram are identity-on-objects, the diagram commutes on objects.  Using the description of the canonical isomorphism $\N \cong U^\N_*\uN$ given at \eqref{eq:can_isos_k}, we observe that the clockwise composite is given on homs as the clockwise composite around the periphery of the following diagram, where $L, M \in \M$; for typographical reasons, we have denoted the $\V$-valued hom bifunctors for $\M, \N$ by $\langle-,-\rangle$ and the internal hom bifunctors by $[-,-]$.  Also, we have denoted the unit objects of $\M$, $\N$ by $I$, $J$, respectively.
$$
\renewcommand{\objectstyle}{\scriptscriptstyle}
\renewcommand{\labelstyle}{\scriptscriptstyle}
\xymatrix@!0@C=23ex @R=8ex{
\langle I,[L,M]\rangle \ar[d]|{L \otimes (-)} \ar[r]^G & \langle GI,G[L,M]\rangle \ar[d]|{GL \otimes (-)} \ar[r]^{\langle e^G,1\rangle} & \langle J,G[L,M]\rangle \ar[d]|{GL \otimes (-)} \ar[r]^{\langle 1,\grave{G}_{LM}\rangle} & \langle J,[GL,GM]\rangle \ar[dd]^\wr\\
\langle L \otimes I,L \otimes [L,M]\rangle \ar@/_10ex/[ddd]^(.7){\langle r^{-1},\Ev\rangle} \ar[d]^G & \langle GL \otimes GI, GL \otimes G[L,M]\rangle \ar[d]|{\langle 1,m^G\rangle} \ar[r]^{\langle 1 \otimes e^G,1\rangle} & \langle GL \otimes J,GL \otimes G[L,M]\rangle \ar[d]|{\langle 1,m^G\rangle} & \\
\langle G(L \otimes I),G(L \otimes [L,M])\rangle \ar[r]^{\langle m^G,1\rangle} \ar[dr]|{\langle Gr^{-1},1\rangle} & \langle GL \otimes GI,G(L \otimes [L,M])\rangle \ar[r]^{\langle 1 \otimes e^G,1\rangle} & \langle GL \otimes J,G(L \otimes [L,M])\rangle \ar[r]^{\langle 1,G\Ev\rangle} & \langle GL \otimes J,GM\rangle \ar[dd]_{\langle r^{-1},1\rangle}\\
& \langle GL,G(L \otimes [L,M])\rangle \ar[ur]|{\langle r,1\rangle} & & \\
\langle L,M\rangle \ar[r]_G & \langle GL,GM\rangle \ar[uurr]|{\langle r,1\rangle} \ar@{=}[rr] & & \langle GL,GM\rangle
}
$$
On the other hand, the counter-clockwise composite $\V$-functor in \eqref{eq:fund_lem_diag} is given on homs by the counter-clockwise composite around the periphery of this diagram, so it suffices to show that the diagram commutes.  The top-left cell commutes by the $\V$-naturality of $m^G$, and the top-right cell commutes by the definition of $\grave{G}_{LM}$ \pbref{def:enr_smc_func}.  The triangular cell on the interior of the diagram commutes by the monoidality of $G$, and the remaining cells in the diagram clearly also commute.
\end{proof}

\begin{DefSub}\label{def:id_on_obs_strsmfun_kg}
Given a 1-cell $G:\V \rightarrow \W$ in $\SMCCAT$, let us denote by
$$K^G:\V \rightarrow (G_*\uV)_0$$
the canonical comparison functor $\theta^G_*\uV$ of \eqref{eq:can_comp}, where here we have identified $(\uV)_0$ with $\V$ by convention.  Recall that $K^G$ is an identity-on-objects strict symmetric monoidal functor.  Note also that $K^G$ is an isomorphism as soon as $G$ is normal.
\end{DefSub}

\begin{CorSub}\label{thm:recovering_smcfunc_from_its_enr}
Given a 1-cell $G:\V \rightarrow \W$ in $\SMCCAT$, the diagram
\begin{equation}\label{eq:recovering_smcfunc_from_its_enr}
\xymatrix{
\V \ar[rr]^{K^G} \ar[dr]_G &     & (G_*\uV)_0 \ar[dl]^{\grave{G}_0}\\
                           & \W  &
}
\end{equation}
in $\CAT$ commutes.
\end{CorSub}
\begin{proof}
Since by definition $K^G = \theta^G_*\uV$ \pbref{def:id_on_obs_strsmfun_kg}, we obtain an instance of the commutative diagram \eqref{eq:fund_lem_diag} in which the base of enrichment is $\SET$.  But since we have identified $\V$ with $U^\V_*\uV = (\uV)_0$ along the canonical isomorphism $U^\V_*\uV \xrightarrow{\sim} \V$, and similarly for $\W$, the left and right sides of the rectangle \eqref{eq:fund_lem_diag} are identity morphisms in this case.
\end{proof}

\begin{PropSub}\label{thm:ggrave_enrsmcfun}
Let $G:\V \rightarrow \W$ be a 1-cell in $\SMCCAT$.  Then $\grave{G}:G_*\uV \rightarrow \uW$ is a 1-cell in $\WSMCCAT$, i.e. a symmetric monoidal $\W$-functor between symmetric monoidal closed $\W$-categories.
\end{PropSub}
\begin{proof}
By \bref{thm:ch_base_vsmccats}, $G_*\uV$ is a symmetric monoidal closed $\W$-category.  The $\W$-functor $\grave{G}$ is given on objects just as $G$, and we claim that the monoidal structure morphisms
$$m^G_{UV}:GU \otimes GV \rightarrow G(U \otimes V)\;\;\;\;(U,V \in \V)$$
and $e^G:I_\W \rightarrow GI_\V$ carried by $G$ constitute a symmetric $\W$-monoidal structure on $\grave{G}$.  

To prove this, we first show that $m^G_{UV}$ is $\W$-natural in $U,V \in G_*\uV$, i.e., that $m^G$ is a 2-cell
\begin{equation}\label{eqn:nat_str_morph_smc_func}
\xymatrix{
G_*\uV \otimes G_*\uV \ar[d]_{\grave{G} \otimes \grave{G}} \ar[r]^{m^{G_*}} & G_*(\uV \otimes \uV) \ar[r]^(.6){G_*(\otimes)} & G_*\uV \ar[d]^{\grave{G}} \ar@{}[dll]^(0.4){}="src"^(0.6){}="tgt" \ar@{<=}"src";"tgt"^{m^G}\\
\uW \otimes \uW \ar[rr]_\otimes & & \uW
}
\end{equation}
in $\WCAT$, where $m^{G_*}$ is the monoidal structure carried by $G_*$ and hence the composite of the upper row is the tensor product $\W$-functor carried by $G_*\uV$.  In detail, we must show that for every pair of objects $(U,V), (U',V')$ in $G_*\uV \otimes G_*\uV$, i.e., all $U,V,U',V' \in \V$, the following diagram in $\W$ commutes 
\begin{equation}\label{eqn:nat_str_morph_smc_func_detail}
\renewcommand{\objectstyle}{\scriptstyle}
\renewcommand{\labelstyle}{\scriptstyle}
\xymatrix{
(G_*\uV \otimes G_*\uV)((U,V),(U',V')) \ar[d]_{m^{G_*}} \ar[rr]^{\grave{G} \otimes \grave{G}} & & (\uW \otimes \uW)((GU,GV),(GU',GV')) \ar[d]^\otimes\\
(G_*(\uV\otimes\uV))((U,V),(U',V')) \ar[d]_{G_*(\otimes)} & & \uW(GU \otimes GV,GU' \otimes GV') \ar[d]^{\uW(1,m^G)}\\
(G_*\uV)(U \otimes V, U' \otimes V') \ar[r]_{\grave{G}} & \uW(G(U \otimes V),G(U' \otimes V'))) \ar[r]_{\uW(m^G,1)} & \uW(GU \otimes GV, G(U' \otimes V'))
}
\end{equation}
where all but the last arrow in each composite is labelled with the name of the $\uW$-functor whose `action-on-homs' is thereby denoted. 

Using the definitions of the $\W$-functors involved in \eqref{eqn:nat_str_morph_smc_func}, we find that the transposes of the composite morphisms in \eqref{eqn:nat_str_morph_smc_func_detail} are given as the two composites on the periphery of the following diagram.  For typographical reasons, we have denoted the internal homs in $\V$ and $\W$ by $[-,-]$.
$$
\renewcommand{\objectstyle}{\scriptscriptstyle}
\renewcommand{\labelstyle}{\scriptscriptstyle}
\xymatrix@!0@C=24ex @R=8ex{
GU \otimes GV \otimes G[U,U'] \otimes G[V,V'] \ar[d]_{m^G \otimes 1 \otimes 1} \ar[drr]^{1 \otimes s \otimes 1} \ar[rrr]^{1 \otimes 1 \otimes \grave{G}_{ UU'} \otimes \grave{G}_{ VV'}} & & & *!<6ex,0ex>{GU \otimes GV \otimes [GU,GU'] \otimes [GV,GV']} \ar[d]^{1 \otimes s \otimes 1}\\
G(U \otimes V) \otimes G[U,U'] \otimes G[V,V'] \ar[dd]_{1 \otimes m^G} & & *!<8.2ex,0ex>{GU \otimes G[U,U'] \otimes GV \otimes G[V,V']} \ar[r]^(.21){1 \otimes \grave{G} \otimes 1 \otimes \grave{G}} \ar@<-4ex>[d]_{m^G \otimes m^G} & *!<6ex,0ex>{GU \otimes [GU,GU'] \otimes GV \otimes [GV,GV']} \ar[d]^{\Ev \otimes \Ev}\\
& & *!<4ex,0ex>{G(U \otimes [U,U']) \otimes G(V \otimes [V,V'])} \ar[r]^(.5){G\Ev \otimes G\Ev} \ar@<-4ex>[d]^{m^G} & GU' \otimes GV' \ar[dd]^{m^G}\\ 
G(U \otimes V) \otimes G([U,U'] \otimes [V,V']) \ar[d]^{G(U \otimes V) \otimes G\left(\otimes_{ (U,V)(U',V')}\right)} \ar[r]^{m^G} & G(U \otimes V \otimes [U,U'] \otimes [V,V']) \ar[d]^{G\left(U \otimes V \otimes (\otimes_{(U,V)(U',V')})\right)} \ar[r]^{G(1 \otimes s \otimes 1)} & G(U \otimes [U,U'] \otimes V \otimes [V,V']) \ar[dr]|{G(\Ev \otimes \Ev)} & \\
G(U \otimes V) \otimes G[U\otimes V, U' \otimes V'] \ar[r]_{m^G} & G(U \otimes V \otimes [U \otimes V,U' \otimes V']) \ar[rr]_{G\Ev} & & G(U' \otimes V')
}
$$
The cell at the bottom-left commutes by the naturality of $m^G$, and the cell immediately to the right of this commutes by the definition of $\otimes_{(U,V)(U',V')}$.  The lowest of the cells on the right-hand-side commutes by the naturality of $m^G$, and the cell immediately above this commutes by the definition of $\grave{G}$.  The cell at the top-right clearly commutes, and the large cell at the top-left is readily shown to commute, by using the fact that $G$ is a \textit{symmetric} monoidal functor.  Hence the diagram commutes, so \eqref{eqn:nat_str_morph_smc_func_detail} commutes.

Having thus shown that $m^G$ is $\W$-natural, it now suffices to show that $\grave{G}:G_*\uV \rightarrow \uW$, when equipped with $m^G$ and $e^G$, satisfies the equations for a symmetric monoidal $\W$-functor.  By \bref{def:id_on_obs_strsmfun_kg}, we have an identity-on-objects strict symmetric monoidal functor $K^G:\V \rightarrow (G_*\uV)_0$, so if $a,\ell,r,s$ denote the symmetric monoidal structure morphisms of $\V$, then those of $G_*\uV$ are obtained as $K^Ga, K^G\ell, K^Gr, K^Gs$.  Hence, using \bref{thm:recovering_smcfunc_from_its_enr}, we readily compute that the diagrammatic equations that must hold in order that $\grave{G}$ be a symmetric monoidal $\W$-functor are exactly the same as those for $G$.
\end{proof}

Knowing now that $\grave{G}$ is a monoidal $\V$-functor, we obtain a strengthened form of the Fundamental Lemma, as follows.

\begin{LemSub}[Monoidal Fundamental Lemma]\label{thm:mon_fund_lemma}
Let $G:\M \rightarrow \N$ be a 1-cell in $\VSMCCAT$.  Then the diagram \eqref{eq:fund_lem_diag} in $\VSMCCAT$ commutes.
\end{LemSub}
\begin{proof}
It suffices to show that the monoidal structures carried by the two composite $\V$-functors in \eqref{eq:fund_lem_diag} are equal.  Per the convention of \bref{par:clsmcat_notn}, we identify $\N_0$ with the underlying ordinary category of $\uN = \underline{\N_0}$.  Then the monoidal structure carried by the $\N_0$-functor $\grave{G}:G_*\uM \rightarrow \uN$ is, by definition, exactly the same as that carried by $G$ \pbref{thm:ggrave_enrsmcfun}.  By \bref{par:ch_base_2func_on_smvcats}, the monoidal structure morphisms carried by the $\V$-functor $U^\N_*(\grave{G})$ can be obtained from those of $\grave{G}$ (i.e., those of $G$) by applying the canonical identity-on-objects functor $\theta^{U^\N}_*:\N_0 = (\uN)_0 \rightarrow (U^{\N}_*\uN)_0$.  However, in view of \bref{rem:can_isos_for_vsmccat}, in fact $(U^{\N}_*\uN)_0 = \N_0$, and the latter functor is the identity functor on $\N_0$.  Hence the monoidal structure carried by $U^\N_*(\grave{G})$ is exactly the same as that of $G$.  The other $\W$-functors in \eqref{eq:fund_lem_diag} are strict monoidal, and by \bref{rem:can_isos_for_vsmccat}, the ordinary functors underlying the vertical arrows in \eqref{eq:fund_lem_diag} are the identity functors on $\M_0$ and $\N_0$, so the result follows.
\end{proof}

\begin{CorSub}\label{thm:recovering_smcfunc_from_its_enr__monoidal_version}
Given a 1-cell $G:\V \rightarrow \W$ in $\SMCCAT$, the diagram \eqref{eq:recovering_smcfunc_from_its_enr} in $\SMCCAT$ commutes.
\end{CorSub}

\section{The 2-functoriality of the autoenrichment}

\begin{ParSub}
Applying the Bakovi\'c-Buckley-Grothendieck construction (\S\bref{par:2groth_constr}) to the 2-functor 
$$\eSMCCAT{(-)}:\SMCCAT \rightarrow \TWOCAT\;,\;\;\;\;\V \mapsto \VSMCCAT$$
of \bref{thm:ch_base_vsmccats}, we obtain 2-functor (indeed, a split op-2-fibration) that we shall denote by
$$P:\ENRSMCCAT \rightarrow \SMCCAT\;.$$
The objects of $\ENRSMCCAT$ are pairs $(\V,\M)$ consisting of a symmetric monoidal closed category $\V$ and a symmetric monoidal closed $\V$-category $\M$, and the fibre of $P$ over $\V$ is isomorphic to the 2-category $\VSMCCAT$.

In the present section, we shall show that there is a 2-functor
$$\underline{(-)}:\SMCCAT \rightarrow \ENRSMCCAT$$
sending each object $\V$ of $\SMCCAT$ to the pair $(\V,\uV)$.
\end{ParSub}

\begin{PropSub} \label{prop:comp_assoc_enr_fun}
Let $G:\U \rightarrow \V$, $H:\V \rightarrow \W$ be symmetric monoidal functors between closed symmetric monoidal categories.  Then the symmetric monoidal $\W$-functor
\begin{equation}\label{eqn:comp_of_assoc_enr_smcfuncs}\widegrave{HG}:(HG)_*\uU \rightarrow \uW\end{equation}
is equal to the composite $H_*G_*\uU \xrightarrow{H_*(\grave{G})} H_*\uV \xrightarrow{\grave{H}} \uW$.
\end{PropSub}
\begin{proof}
That the underlying $\W$-functors of these two monoidal $\W$-functors are equal follows from \cite[I 6.6]{EiKe}, so it remains only to show that their monoidal structures coincide.  By definition, the monoidal structure carried by $\grave{G}$ (resp. $\grave{H}$, $\widegrave{HG}$) is exactly the same as that of $G$ (resp. $H$, $HG$).  Also, by \bref{par:ch_base_2func_on_smvcats}, the monoidal structure morphisms carried by $H_*(\grave{G})$ are obtained from those of $\grave{G}$ (i.e., those of $G$) by applying the canonical comparison functor $K^H:\V \rightarrow (H_*\uV)_0$.  Hence the monoidal structure carried by the composite $\grave{H} \circ H_*(\grave{G})$ consists of the composites
$$I_\W \xrightarrow{e^H} HI_\V \xrightarrow{\grave{H}K^He^G} HGI_\U$$
$$HGU \otimes HGU' \xrightarrow{m^H} H(GU \otimes GU') \xrightarrow{\grave{H}K^Hm^G} HG(U \otimes U')\;.$$
But by \bref{thm:recovering_smcfunc_from_its_enr}, $\grave{H}K^He^G = He^G$ and $\grave{H}K^Hm^G = Hm^G$, so the latter composites are equally the monoidal structure morphisms carried by $HG$, which are the same as those carried by $\widegrave{HG}$. 
\end{proof}

\begin{LemSub}\label{thm:lemma_for_self_enr_on_2-cells}
Given a 2-cell $\alpha:G \Rightarrow H:\W \rightarrow \V$ in $\SMCCAT$, the components of $\alpha$ constitute a 2-cell
\begin{equation}\label{eq:lemma_for_self_enr_on_2-cells}
\xymatrix{
G_*\uW \ar[dr]_{\grave{G}}^{}="src1" \ar[rr]^{\alpha_*\uW} &         & H_*\uW \ar@{}"src1";{}|(.2){}="tgt2"|(.5){}="src2" \ar[dl]^{\grave{H}} \\
                                                         & \uV     &
  \ar@{<=}"src2";"tgt2"^{\grave{\alpha}}
}
\end{equation}
in $\VSMCCAT$. 
\end{LemSub}
\begin{proof}
Firstly, to show that the components of $\alpha$ constitute a $\V$-natural transformation $\grave{\alpha}$, we must show that for each pair of objects $W,X \in \W$, the diagram
$$
\xymatrix{
G\uW(W,X) \ar[d]_{\grave{G}_{WX}} \ar[r]^{\alpha_{\uW(W,X)}} & H\uW(W,X) \ar[r]^{\grave{H}_{WX}} & \uV(HW,HX) \ar[d]^{\uV(\alpha_W,1)}\\
\uV(GW,GX) \ar[rr]_{\uV(1,\alpha_X)} & & \uV(GW,HX)
}
$$
commutes.  By the definition of $\grave{G}$ and $\grave{H}$, the transposes of the two composites in this diagram are obtained as the two composites around the periphery of the following diagram
$$
\xymatrix{
GW\otimes G\uW(W,X) \ar[r]^{m^G} \ar[d]_{\alpha\otimes\alpha} & G(W\otimes\uW(W,X)) \ar[d]^\alpha \ar[r]^(.7){G\Ev} & GX \ar[d]^\alpha\\
HW\otimes H\uW(W,X) \ar[r]_{m^H} & H(W\otimes\uW(W,X)) \ar[r]_(.7){H\Ev} & HX
}
$$
which commutes since $\alpha$ is a monoidal natural transformation.  It remains only to show that the resulting $\V$-natural transformation $\grave{\alpha}$ is monoidal, but by definition the monoidal structures carried by $\grave{G}$ and $\grave{H}$ are the same as those carried by $G$ and $H$, and $\alpha_*\uW$ is strict symmetric monoidal and identity-on-objects, so the diagrammatic equations expressing the monoidality of $\grave{\alpha}$ are exactly the same as those for $\alpha$ itself.
\end{proof}

\begin{ThmSub}\label{thm:2func_self_enr}
There is a 2-functor
$$\underline{(-)}:\SMCCAT \rightarrow \ENRSMCCAT$$
sending each object $\V$ of $\SMCCAT$ to the pair $(\V,\uV)$, and sending each 1-cell $G:\W \rightarrow \V$ in $\SMCCAT$ to $(G,\grave{G}):(\W,\uW) \rightarrow (\V,\uV)$, where $\grave{G}:G_*(\uW) \rightarrow \uV$ is the symmetric monoidal $\V$-functor defined in \S \bref{sec:enr_smcfunc}.
\end{ThmSub}
\begin{proof}
By \bref{prop:comp_assoc_enr_fun}, the given assignment on 1-cells preserves composition, and it is immediate that it preserves identity 1-cells and so is functorial.  Given a 2-cell $\alpha:G \Rightarrow H:\W \rightarrow \V$ in $\K := \SMCCAT$, the associated 2-cell $\grave{\alpha}$ given in \bref{thm:lemma_for_self_enr_on_2-cells} yields an associated 2-cell $(\alpha,\grave{\alpha}):(\W,\uW) \rightarrow (\V,\uV)$ in $\F := \ENRSMCCAT$.  The components of $\grave{\alpha}$ are the same as those of $\alpha$, and the $\V$-functor $\alpha_*\uW$ appearing in \eqref{eq:lemma_for_self_enr_on_2-cells} is identity-on-objects, so with reference to the definition of vertical composition in $\F$, it is immediate that the given assignment on 2-cells preserves vertical composition, and it clearly preserves identity 2-cells.  Hence it suffices to show that it preserves whiskering.  Firstly, let
\begin{equation}\label{eqn:whiskering_1}
\xymatrix{
\W \ar[r]^F & \V \ar@/^1.5ex/[rr]^G="s1" \ar@/_1.5ex/[rr]_H="t1" & & \U
  \ar@{}"s1";"t1"|(.35){}="s2"|(.65){}="t2"
  \ar@{=>}"s2";"t2"^\alpha
}
\end{equation}
in $\K$.  On the one hand, we can take the composite 2-cell $\alpha F$ in $\K$ and then apply $\underline{(-)}$ to obtain the 2-cell $(\alpha F,\widegrave{\alpha F}):(GF,\widegrave{GF}) \Rightarrow (HF,\widegrave{HF}):(\W,\uW) \rightarrow (\U,\uU)$ in $\F$.  On the other hand, we can apply $\underline{(-)}$ and then take the composite 2-cell
\begin{equation}\label{eqn:whiskering_2}
\xymatrix{
(\W,\uW) \ar[r]^{(F,\grave{F})} & (\V,\uV) \ar@/^1.5ex/[rr]^{(G,\grave{G})}="s1" \ar@/_1.5ex/[rr]_{(H,\grave{H})}="t1" & & (\U,\uU)
  \ar@{}"s1";"t1"|(.35){}="s2"|(.65){}="t2"
  \ar@{=>}"s2";"t2"^{(\alpha,\grave{\alpha})}
}
\end{equation}
in $\F$, which by definition consists of the 2-cell $\alpha F:GF \Rightarrow HF$ in $\K$ together with the composite 2-cell
$$
\xymatrix{
G_*F_*(\uW) \ar[d]_{(\alpha F)_*\uW\:=\:\alpha_*F_*\uW} \ar[r]^{G_*(\grave{F})} & G_*(\uV) \ar@{}[dl]|(.45){}="s3"|(.55){}="t3" \ar[d]_{\alpha_* \uV} \ar[drr]^{\grave{G}}="s1" & & \\
H_*F_*(\uW) \ar[r]_{H_*(\grave{F})} & H_*(\uV) \ar[rr]_{\grave{H}} \ar@{}"s1";{}|(.4){}="s2"|(.6){}="t2" & & \uU
  \ar@{=>}"s2";"t2"_{\grave{\alpha}}
  \ar@{=}"s3";"t3"
}
$$
in $\eSMCCAT{\U}$.  But the components of $\grave{\alpha}$ are exactly those of $\alpha$ itself, so since $G_*(\grave{F})$ is given as $F$ on objects, the components of the latter composite 2-cell are the same as those of $\alpha F$, which are equally those of $\widegrave{\alpha F}$.  Hence the composite 2-cell \eqref{eqn:whiskering_2} is equal to $(\alpha F,\widegrave{\alpha F})$, as needed.

Next, let
$$
\xymatrix{
\W \ar@/^1.5ex/[rr]^F="s1" \ar@/_1.5ex/[rr]_G="t1" & & \V \ar[r]^H & \U
  \ar@{}"s1";"t1"|(.35){}="s2"|(.65){}="t2"
  \ar@{=>}"s2";"t2"^\alpha
}
$$
in $\K$.  On the one hand, we can take the composite 2-cell $H\alpha$ in $\K$, and then apply $\underline{(-)}$ to obtain a 2-cell $(H\alpha,\widegrave{H\alpha}):(HF,\widegrave{HF}) \Rightarrow (HG,\widegrave{HG}):(\W,\uW) \rightarrow (\U,\uU)$ in $\F$.  On the other hand, we can take the composite 2-cell
\begin{equation}\label{eqn:whiskering_3}
\xymatrix{
(\W,\uW) \ar@/^1.5ex/[rr]^{(F,\grave{F})}="s1" \ar@/_1.5ex/[rr]_{(G,\grave{G})}="t1" & & (\V,\uV) \ar[r]^{(H,\grave{H})} & (\U,\uU)
  \ar@{}"s1";"t1"|(.35){}="s2"|(.65){}="t2"
  \ar@{=>}"s2";"t2"^{(\alpha,\grave{\alpha})}
}
\end{equation}
in $\F$, which consists of the composite 2-cell $H\alpha$ in $\K$ together with the composite 2-cell
\begin{equation}\label{eqn:whiskering_4}
\xymatrix{
H_*F_*(\uW) \ar[d]_{(H\alpha)_*\uW\:=\:H_*\alpha_*\uW} \ar[drr]^{H_*(\grave{F})}="s1" & & & \\
H_*G_*(\uW) \ar[rr]_{H_*(\grave{G})} \ar@{}"s1";{}|(.4){}="s2"|(.6){}="t2" & & H_*(\uV) \ar[r]^{\grave{H}} & \uU
  \ar@{=>}"s2";"t2"_{H_*(\grave{\alpha})}
}
\end{equation}
in $\eSMCCAT{\U}$.  The component of this 2-cell at each object $W$ of $H_*F_*(\uW)$ (equivalently, of $\W$) is obtained by applying $\grave{H}$ to the component of $H_*(\grave{\alpha})$ at $W$, which in turn is obtained by applying $K^H:\V \rightarrow (H_*\uV)_0$ \pbref{def:id_on_obs_strsmfun_kg} to that of $\alpha$ (by \bref{par:ch_base_2func_on_smvcats}).  Hence, using \bref{thm:recovering_smcfunc_from_its_enr}, the resulting component of the 2-cell $\eqref{eqn:whiskering_4}$ at $W$ is
$$\grave{H}K^H\alpha_W = H\alpha_W = \widegrave{H\alpha}_W\;.$$
Hence the 2-cell \eqref{eqn:whiskering_3} is equal to $(H\alpha,\widegrave{H\alpha})$, as needed.
\end{proof}

\section{The 2-functoriality of the induced enrichment over a fixed base}

Whereas the previous section shows that the autoenrichment $\V \mapsto \uV$ extends to a 2-functor $\underline{(-)}:\SMCCAT \rightarrow \ESMCCAT$ valued in the 2-category of symmetric monoidal closed categories enriched over various bases, we show in the present section that for a fixed base $\V$, there is an induced 2-functor
$$\Enr_\V:\SMCCAT\sslash\V \rightarrow \VSMCCAT\sslash\uV$$
between the \textit{lax slice 2-categories} \pbref{def:lax_slice} over $\V$.  Whereas it is immediate that the autoenrichment 2-functor induces a 2-functor $\SMCCAT\sslash \uV \rightarrow (\ESMCCAT)\sslash \uV$ valued in a lax slice of $\ESMCCAT$, the challenge that remains is to (2-functorially) map the latter lax slice into the lax slice over $\uV$ \textit{in the fibre} $\VSMCCAT$.  This we accomplish by means of a general lemma on split op-2-fibrations \pbref{thm:func_from_laxsl_in_splop2fibr_to_laxsl_in_fibre}.

\begin{DefSub}\label{def:lax_slice}
The \textit{lax slice 2-category} $\K \sslash K$ over an object $K$ of a 2-category $\K$ has objects all 1-cells $g:L \rightarrow K$ with codomain $K$, written as pairs $(L,g)$.  A 1-cell $(s,\sigma):(L,g) \rightarrow (M,h)$ in $\K \sslash K$ consists of a 1-cell $s:L \rightarrow M$ and a 2-cell $\sigma:g \Rightarrow hs$.  A 2-cell $\alpha:(s,\sigma) \Rightarrow (t,\tau):(L,g) \rightarrow (M,h)$ is simply a 2-cell $\alpha:s \Rightarrow t$ such that $\alpha \circ \sigma = \tau$, where $\alpha \circ \sigma$ denotes the pasted 2-cell $h\alpha \cdot \sigma$.
\end{DefSub}

\begin{LemSub}\label{thm:func_from_laxsl_in_splop2fibr_to_laxsl_in_fibre}
Let $P:\F \rightarrow \K$ be a split op-2-fibration, and let $A$ be an object of $\F$.  Then there is an associated 2-functor
$$\F \sslash A \rightarrow \F_{PA} \sslash A$$
from the lax slice $\F \sslash A$ over $A$ in $\F$ to the lax slice $\F_{PA} \sslash A$ over $A$ in the fibre $\F_{PA}$ of $P$ over $PA$.
\end{LemSub}
\begin{proof}
Given an object $(B,f)$ of $\F \sslash A$, so that $f:B \rightarrow A$ is a 1-cell in $\F$, we have a cocartesian 1-cell $\psi(Pf,B):B \rightarrow (Pf)_*(B)$ over $Pf$ and an extension problem $(\psi(Pf,B),f,1_{PA})$, which therefore has a unique solution $f'$, i.e. a 1-cell $f':(Pf)_*(B) \rightarrow A$ in the fibre over $PA$ such that $f' \circ \psi(Pf,B) = f$.  The 2-functor to be defined shall send the object $(B,f)$ to the object $((Pf)_*(B),f')$ of $\F_{PA} \sslash A$.

Given a 1-cell $(s,\beta):(B,f) \rightarrow (C,g)$ in $\F \sslash A$, so that $s:B \rightarrow C$ is a 1-cell and $\beta:f \Rightarrow gs$ a 2-cell, we claim that there are unique $\beta_0$ and $s_\beta$ as in the left half of the diagram
$$
\xymatrix{
B \ar[d]_s \ar[r]^(.35){\psi(Pf,B)} & (Pf)_*(B) \ar[d]_{s_\beta} \ar@/^2ex/[dr]^{f'}="s2" \ar@{}[dl]|(.4){}="s1"|(.6){}="t1" & \\
C \ar[r]_(.35){\psi(Pg,C)} & (Pg)_*(C) \ar[r]_{g'} \ar@{}"s2";{}|(.4){}="s3"|(.6){}="t3" & A
  \ar@{=>}"s1";"t1"_{\beta_0}
  \ar@{=>}"s3";"t3"^{\beta'}
}
$$
such that
\begin{enumerate}
\item $\beta_0$ is a designated cartesian 2-cell over $P\beta$, and
\item $s_\beta$ lies in $\F_{PA}$.
\end{enumerate}
Indeed, letting $q := \psi(Pg,C) \circ s$ and taking $\beta_0 := \varphi(P\beta,q):(P\beta)^*(q) \Rightarrow q$, we have an extension problem $(\psi(Pf,B),(P\beta)^*(q),1_{PA})$, which has a unique solution $s_\beta$, so the claim is proved.  Further, there is a unique 2-cell $\beta'$ as in the diagram such that 
\begin{enumerate}
\item[3.] $\beta'$ lies in $\F_{PA}$, and
\item[4.] the pasted 2-cell equals $\beta$.
\end{enumerate}
Indeed, we have a lifting problem $(g' \circ \beta_0,\beta,1_{Pf})$, but since $\beta_0$ is cartesian, the whiskered 2-cell $g' \circ \beta_0$ is cartesian, so this lifting problem has a unique solution $\tau:f \Rightarrow g' \circ (P\beta)^*(q)$.  But then we have an extension problem $(\psi(Pf,B),\tau, 1_{1_{PA}})$ with a unique solution $\beta':u \Rightarrow v:(Pf)_*(B) \rightarrow A$.  Hence $u$ and $v$ are necessarily solutions to the lifting problems $(\psi(Pf,B),f,1_{PA})$ and $(\psi(Pf,B),g' \circ (P\beta)^*(q),1_{PA})$, respectively, but these also have solutions $f'$ and $g' \circ s_\beta$, so in fact $u = f'$ and $v = g' \circ s_\beta$.  One readily finds that $\beta'$ is equivalently characterized by 3 and 4.  We define the needed 2-functor on 1-cells by sending $(s,\beta)$ to $(s_\beta,\beta'):((Pf)_*(B),f') \rightarrow ((Pg)_*(C),g')$.

Given a 2-cell $\alpha:(s,\beta) \Rightarrow (t,\gamma):(B,f) \rightarrow (C,g)$ in $\F \sslash A$, so that $\alpha:s \Rightarrow t$ and the 2-cell obtained by pasting $\alpha$ and $\beta$ equals $\gamma$, we claim that there is a unique 2-cell $\alpha_{\gamma\beta}$ in $\F_{PA}$ such that the following pasted 2-cells are equal
$$
\xymatrix{
B \ar@/_2ex/[d]_t \ar[r]^(.35){\psi(Pf,B)} & (Pf)_*(B) \ar@{}[dl]|(.5){}="s1"|(.7){}="t1" \ar@/_2ex/[d]_{t_\gamma}="t2" \ar@/^2ex/[d]^{s_\beta}="s2" \\
C \ar[r]_(.35){\psi(Pg,C)} & (Pg)_*(C)
  \ar@{=>}"s1";"t1"_{\gamma_0}
  \ar@{}"s2";"t2"|(.4){}="s3"|(.6){}="t3"
  \ar@{=>}"s3";"t3"_{\alpha_{\gamma\beta}}
}\;\;\;\;
\xymatrix{
B \ar@/_2ex/[d]_t="t2" \ar@/^2ex/[d]^s="s2" \ar[r]^(.35){\psi(Pf,B)} & (Pf)_*(B) \ar@{}[dl]|(.4){}="s1"|(.6){}="t1" \ar@/^2ex/[d]^{s_\beta} \\
C \ar[r]_(.35){\psi(Pg,C)} & (Pg)_*(C)
  \ar@{=>}"s1";"t1"_{\beta_0}
  \ar@{}"s2";"t2"|(.4){}="s3"|(.6){}="t3"
  \ar@{=>}"s3";"t3"_{\alpha}
}
$$
Indeed, letting $\rho$ denote the rightmost of these composite 2-cells, we have a lifting problem $(\gamma_0,\rho,1_{Pf})$ with a unique solution $\alpha':s_\beta \circ \psi(Pf,B) \Rightarrow t_\gamma \circ \psi(Pf,B)$, and we then obtain an extension problem $(\psi(Pf,B),\alpha',1_{1_{PA}})$ with a unique solution $\alpha_{\gamma\beta}$, which is equally the unique 2-cell in $\F_{PA}$ making the given pasted 2-cells equal.  We define the needed 2-functor on 2-cells by sending $\alpha:(s,\beta) \Rightarrow (t,\gamma)$ to $\alpha_{\gamma\beta}:(s_\beta,\beta') \Rightarrow (t_\gamma,\gamma')$.  One readily verifies that the latter is indeed a 2-cell in $\F_{PA}\sslash A$ by using the defining properties of $\alpha_{\gamma\beta}$, $\beta'$, and $\gamma'$ to show that the 2-cell obtained by pasting $\alpha_{\gamma\beta}$ with $\beta'$ equals $\gamma'$.

Using the fact that the designated cartesian 2-cells are closed under pasting, one now readily shows that the given assignment preserves composition of 1-cells, identity 1-cells, vertical composition of 2-cells, identity 2-cells, and whiskering, and hence is 2-functorial. 
\end{proof}

\begin{ParSub}\label{par:func_from_laxsl_in_grothconstr_to_laxsl_in_fibre}
In the case that $\F$ is the split op-2-fibration $\int\Phi$ associated to a 2-functor $\Phi:\K \rightarrow \TWOCAT$, the preceding proposition \pbref{thm:func_from_laxsl_in_splop2fibr_to_laxsl_in_fibre} yields a composite 2-functor
$$\F \sslash A \longrightarrow \F_{A^{\tinyminus}} \sslash A \overset{\sim}{\longrightarrow} \Phi(A^{\tinyminus})\sslash A^{\tinyplus}$$
associated to each object $A = (A^{\tinyminus},A^{\tinyplus})$ of $\int\Phi$.  We can describe this 2-functor explicitly in terms of $\Phi$, as follows.  It sends each object $(B,f)$ of its domain to $(f^{\tinyminus}_*(B^{\tinyplus}),f^{\tinyplus})$, each 1-cell $(s,\beta):(B,f) \rightarrow (C,g)$ to $(g^{\tinyminus}_*(s^{\tinyplus}) \circ \beta^{\tinyminus}_*B^{\tinyplus},\beta^{\tinyplus}):(f^{\tinyminus}_*(B^{\tinyplus}),f^{\tinyplus}) \rightarrow (g^{\tinyminus}_*(C^{\tinyplus}),g^{\tinyplus})$, and each 2-cell $\alpha:(s,\beta) \rightarrow (t,\gamma)$ to $g^{\tinyminus}_*(\alpha^{\tinyplus}) \circ \beta^{\tinyminus}_*B^{\tinyplus}:(g^{\tinyminus}_*(s^{\tinyplus}) \circ \beta^{\tinyminus}_*B^{\tinyplus},\beta^{\tinyplus}) \Rightarrow (g^{\tinyminus}_*(t^{\tinyplus}) \circ \gamma^{\tinyminus}_*B^{\tinyplus},\gamma^{\tinyplus})$.
\end{ParSub}

\begin{ThmSub}\label{thm:enr_v}
Given a symmetric monoidal closed category $\V$, there is a 2-functor
$$\Enr_\V:\SMCCAT \sslash \V \rightarrow \VSMCCAT \sslash \uV$$
given on objects by sending $G:\M \rightarrow \V$ to $\grave{G}:G_*\uM \rightarrow \uV$. 
\end{ThmSub}
\begin{proof}
The 2-functor $\underline{(-)}:\SMCCAT \rightarrow \ENRSMCCAT$ of \bref{thm:2func_self_enr} induces a 2-functor 
$$\SMCCAT \sslash \V \rightarrow \left(\ENRSMCCAT\right) \sslash (\V,\uV)$$
between the lax slices.  Hence by \bref{par:func_from_laxsl_in_grothconstr_to_laxsl_in_fibre} we obtain a composite 2-functor
$$\SMCCAT \sslash \V \longrightarrow \left(\ENRSMCCAT\right) \sslash (\V,\uV) \longrightarrow \VSMCCAT \sslash \uV\;.$$
\end{proof}

\begin{RemSub}\label{rem:expl_desc_enrv}
Explicitly, the 2-functor $\Enr_\V$ of \bref{thm:enr_v} is given on 1-cells as follows.  Given a 1-cell $(S,\beta):(\M,G) \rightarrow (\N,H)$ in the lax slice 2-category $\SMCCAT \sslash \V$, i.e.
$$
\xymatrix{
\M \ar[dr]_G^{}="src1" \ar[rr]^S &         & \N \ar@{}"src1";{}|(.2){}="tgt2"|(.5){}="src2" \ar[dl]^H \\
                                                         & \V     &
  \ar@{<=}"src2";"tgt2"^\beta
}
$$
in $\SMCCAT$, the associated 1-cell $(G_*\uM,\grave{G}) \rightarrow (H_*\N,\grave{H})$ in $\VSMCCAT \sslash \uV$ is given by the following diagram in $\VSMCCAT$.
$$
\xymatrix{
G_*\uM \ar@/_1.5ex/[dr]_{\grave{G}}|{}="s1" \ar[r]^{\beta_*\uM} & H_*S_*\uM \ar@{}"s1";{}|(.2){}="s2"|(.6){}="t2" \ar[d]^(.4){\widegrave{HS}}|(.6){}="s3" \ar[r]^{H_*(\grave{S})} & H_*\uN \ar@{}"s3";{}|(.45){}="s4"|(.55){}="t4" \ar@/^1.5ex/[dl]^{\grave{H}} \\
& \uV & 
  \ar@{=>}"s2";"t2"_{\grave{\beta}}
  \ar@{=}"s4";"t4"
}
$$
Given a 2-cell $\alpha:(S,\beta) \Rightarrow (T,\gamma):(\M,G) \rightarrow (\N,H)$ in $\SMCCAT \sslash \V$, the associated 2-cell in $\VSMCCAT \sslash \uV$ is $H_*(\grave{\alpha}) \circ \beta_*\uM$.
\end{RemSub}

\section{Enrichment of a symmetric monoidal closed adjunction}

\begin{LemSub}\label{thm:adj_dets_adj_in_laxslice}
Let $f \nsststile{\varepsilon}{\eta} g:B \rightarrow A$ be an adjunction in a 2-category $\K$.  Then $(f,\eta) \nsststile{\varepsilon}{\eta} (g,1_g):(B,g) \rightarrow (A,1_A)$ is an adjunction in the lax slice 2-category $\K\sslash A$.
\end{LemSub}
\begin{proof}
The verification is straightforward.
\end{proof}

\begin{ThmSub}\label{thm:enr_smcadj}
Let $F \nsststile{\varepsilon}{\eta} G:\M \rightarrow \V$ be an adjunction in $\SMCCAT$.  Then there is an associated adjunction
$$\acute{F} \nsststile{\varepsilon}{\eta} \grave{G}:G_*\uM \rightarrow \uV$$
in $\VSMCCAT$ whose underlying adjunction in $\SMCCAT$ may be identified with the given adjunction, via an isomorphism $\M \cong (G_*\uM)_0$.
\end{ThmSub}
\begin{proof}
By \bref{thm:adj_dets_adj_in_laxslice}, we have an adjunction $(F,\eta) \nsststile{\varepsilon}{\eta} (G,1_G):(\M,G) \rightarrow (\V,1_{\V})$ in $\SMCCAT\sslash\V$.  The composite 2-functor
\begin{equation}\label{eq:2func_used_for_enrmt_adj}\SMCCAT\sslash\V \xrightarrow{\Enr_\V} \VSMCCAT\sslash\uV \xrightarrow{\Dom} \VSMCCAT\end{equation}
(where $\Dom$ is the `domain' 2-functor) sends this adjunction to an adjunction in $\VSMCCAT$ that we now describe explicitly.  

Using the description of $\Enr_\V$ given in \bref{rem:expl_desc_enrv}, we find that whereas the right adjoint is $\grave{G}$ \pbref{def:enr_smc_func}, the left adjoint $\acute{F}$ is the composite
$$\uV \xrightarrow{\eta_*\uV} G_*F_*\uV \xrightarrow{G_*(\grave{F})} G_*\uM\;.$$
The unit is obtained by sending $\eta:1_{(\V,1)} \Rightarrow (G,1)(F,\eta) = (GF,\eta)$ along \eqref{eq:2func_used_for_enrmt_adj} to yield the 2-cell
$$
\xymatrix{
\uV \ar[d]_{\eta_*\uV} \ar[dr]^1="s1" & \\
G_*F_*\uV \ar[r]_{\widegrave{GF}} \ar@{}"s1";{}|(.4){}="s2"|(.6){}="t2" & \uV
  \ar@{=>}"s2";"t2"_{\grave{\eta}}
}
$$
in $\VSMCCAT$ \pbref{thm:lemma_for_self_enr_on_2-cells}, whose components are equally those of $\eta$.  The counit $\acute{\varepsilon}$ is obtained by sending $\varepsilon:(F,\eta)(G,1) = (FG,\eta G) \Rightarrow 1_{(\M,G)}$ along \eqref{eq:2func_used_for_enrmt_adj} to yield the 2-cell
\begin{equation}\label{eqn:eps_acute}
\xymatrix{
&  G_*F_*G_*\uM \ar@{}[l]|(.0){}="s3" \ar[d]|{G_*\varepsilon_*\uM} \ar@/^1.5ex/[dr]^{G_*(\widegrave{FG})}="s1" &\\
G_*\uM \ar@/^1.5ex/[ur]^{\eta_*G_*\uM} \ar[r]_1|(.3){}="t3" & G_*\uM \ar[r]_1 \ar@{}"s1";{}|(.4){}="s2"|(.6){}="t2" & G_*\uM
  \ar@{=>}"s2";"t2"^{G_*(\grave{\varepsilon})}
  \ar@{}"s3";"t3"|(.5){}="s4"|(.6){}="t4"
  \ar@{=}"s4";"t4"
}
\end{equation}
which we denote by $\acute{\varepsilon}$.  We already know that $\grave{\eta}$ and $\acute{\varepsilon}$ have the expected domain and codomain, even though this may not be immediately evident from the diagrams.

By \cite[Proposition 2.1]{Ke:Doctr}, the right adjoint 1-cell $G$ in $\SMCCAT$ is normal, so the comparison 1-cell $K^G:\M \rightarrow (G_*\uM)_0$ in $\SMCCAT$ is an isomorphism \pbref{def:id_on_obs_strsmfun_kg}.  Since $K^G$ commutes with $G$ and $\grave{G}_0:(G_*\uM)_0 \rightarrow \V$ in $\SMCCAT$ \pbref{thm:recovering_smcfunc_from_its_enr__monoidal_version}, the monoidal functors $G$ and $\grave{G}_0$ in $\SMCCAT$ are identified as soon as we identify $\M$ with $(G_*\uM)_0$ along $K^G$.  Once we do so, we have an adjunction 
$\acute{F}_0 \nsststile{\acute{\varepsilon}_0}{\grave{\eta_0}} \grave{G}_0 = G:\M \rightarrow \V$ in $\SMCCAT$ that we may compare with $F \nsststile{\varepsilon}{\eta} G$.  Such an adjunction is uniquely determined by the right adjoint $G$ together with the family $(\acute{F}V,\grave{\eta}_V) = (FV,\eta_V)$ indexed by the objects $V \in \V$; indeed, for adjunctions in $\CAT$ this is well-known, and for adjunctions in $\SMCCAT$ it follows from Kelly's work on doctrinal adjunction (\cite[1.4]{Ke:Doctr}).  Hence the two adjunctions in question are identical.
\end{proof}

\begin{RemSub}\label{rem:enr_smcadj}
In the situation of \bref{thm:enr_smcadj}, the given adjunction is sent by $\eSMCCAT{(-)}:\SMCCAT \rightarrow \TWOCAT$ to a 2-adjunction 
$$\label{eq:2adj_chbase}F_* \nsststile{\varepsilon_*}{\eta_*} G_*:\eSMCCAT{\M} \rightarrow \VSMCCAT\;.$$
Notice that $\acute{F}:\uV \rightarrow G_*\uM$ is the transpose of $\grave{F}:F_*\uV \rightarrow \uM$ under this 2-adjunction, and with reference to \eqref{eqn:eps_acute}, $\acute{\varepsilon}$ is the transpose of the 2-cell $\grave{\varepsilon}$.
\end{RemSub}

\begin{ExaSub}\label{exa:smc_adj_sheaves}
Letting $g:X \rightarrow S$ be a morphism of schemes, and letting ($\M$,$\V$) be any one of the pairs of categories (i), (ii), or (iii) listed in \S \bref{sec:intro}, the associated adjunction $g^* \dashv g_*:\M \rightarrow \V$ is symmetric monoidal closed, as we now verify.

For (ii) and (iii), this is all but made explicit in \cite{Lip}, so we piece together the needed facts here.  In cases (ii) and (iii), $g_*$ is a symmetric monoidal functor by, e.g., \cite[3.4.4]{Lip}.  By \cite[1.1, 1.2]{Ke:Doctr}, the \textit{mate} of the monoidal structure on $g_*$ is an \textit{op-monoidal} structure on $g^*$, consisting of morphisms
$$g^*(\cO_S) \rightarrow \cO_X\;\;\;\;\;\;\;\;\;\;\;\;g^*(V \otimes W) \rightarrow g^*(V)\otimes g^*(W)\;\;\;\;(V,W \in \V).$$
But as noted in \cite[3.5.4, 3.2.4, 3.1.9]{Lip}, the rightmost is an isomorphism.  It is well-known that the leftmost is also an isomorphism; indeed, in case (ii) it is simply the canonical isomorphism $g^*(\cO_S) = \cO_X \otimes_{g^{-1}(\cO_S)} g^{-1}(\cO_S) \xrightarrow{\sim} \cO_X$, and in case (iii) it is obtained from the latter isomorphism via \cite[3.2.5(a)]{Lip} and so is an isomorphism in $\text{D}(\Mod{\cO_X})$ since $\cO_S$ is \textit{q-flat} as a complex\footnote{In more detail, since $\cO_S$ is q-flat, the definition of the left-derived functor $\textsf{\textbf{L}}g^*$ in terms of q-flat resolutions \cite[3.1]{Lip} entails that $\textsf{\textbf{L}}g^*(\cO_S) = g^*(\cO_S)$ where $g^*$ denotes the inverse image of sheaves of modules.  Moreover, since the canonical morphism $\textsf{\textbf{L}}g^*(\cO_S) \rightarrow g^*(\cO_S)$ is simply the identity morphism, the morphism $\textsf{\textbf{L}}g^*(\cO_S) \rightarrow \cO_X$ in question is therefore merely the isomorphism of $\cO_X$-modules $g^*(\cO_S) \rightarrow \cO_X$, considered as a morphism in $\text{D}(\Mod{\cO_X})$.} (by \cite[2.5.2]{Lip}, as $\cO_S$ is a flat $\cO_S$-module).  Hence \cite[1.4]{Ke:Doctr} applies.

Regarding (i), it is well-known that if $g^* \dashv g_*:\M \rightarrow \V$ is an adjunction between categories with finite products, and $g^*$ preserves finite products, then this adjunction is symmetric monoidal; indeed, the mate of the cartesian monoidal structure on $g_*$ consists of isomorphisms, so again \cite[1.4]{Ke:Doctr} applies.
\end{ExaSub}

\bibliographystyle{amsplain}
\bibliography{bib}

\providecommand{\bysame}{\leavevmode\hbox to3em{\hrulefill}\thinspace}
\providecommand{\MR}{\relax\ifhmode\unskip\space\fi MR }
\providecommand{\MRhref}[2]{%
  \href{http://www.ams.org/mathscinet-getitem?mr=#1}{#2}
}
\providecommand{\href}[2]{#2}
\begin{thebibliography}{10}

\bibitem{BooSt}
T.~Booker and R.~Street, \emph{Tannaka duality and convolution for duoidal
  categories}, Theory Appl. Categ. \textbf{28} (2013), No. 6, 166--205.

\bibitem{Buck}
M.~Buckley, \emph{Fibred 2-categories and bicategories}, J. Pure Appl. Algebra
  \textbf{218} (2014), no.~6, 1034--1074.

\bibitem{Cr}
G.~S.~H. Cruttwell, \emph{Normed spaces and the change of base for enriched
  categories}, Ph.D. thesis, Dalhousie University, 2008.

\bibitem{DayStr}
B.~Day and R.~Street, \emph{Monoidal bicategories and {H}opf algebroids}, Adv.
  Math. \textbf{129} (1997), no.~1, 99--157.

\bibitem{EiKe}
S.~Eilenberg and G.~M. Kelly, \emph{Closed categories}, Proc. {C}onf.
  {C}ategorical {A}lgebra ({L}a {J}olla, {C}alif., 1965), Springer, 1966,
  pp.~421--562.

\bibitem{GPS}
R.~Gordon, A.~J. Power, and R.~Street, \emph{Coherence for tricategories}, Mem.
  Amer. Math. Soc. \textbf{117} (1995), no.~558.

\bibitem{Gr:Fibred}
A.~Grothendieck, \emph{Cat\'egories fibr\'ees et descente}, Rev\^etements
  \'etales et groupe fondamental, S\'eminaire de G\'eom\'etrie Alg\'ebrique du
  Bois-Marie 1960/61 (SGA 1), Expos\'e VI, 1963, reprinted in Lecture Notes in
  Mathematics 224, Springer-Verlag, 1971, 145--194.

\bibitem{Ke:EnrAdj}
G.~M. Kelly, \emph{Adjunction for enriched categories}, Lecture Notes in Math.
  \textbf{106} (1969), 166--177.

\bibitem{Ke:Doctr}
\bysame, \emph{Doctrinal adjunction}, Lecture Notes in Math. \textbf{420}
  (1974), 257--280.

\bibitem{Ke:Ba}
\bysame, \emph{Basic concepts of enriched category theory}, Repr. Theory Appl.
  Categ. (2005), no.~10, Reprint of the 1982 original [Cambridge Univ. Press].

\bibitem{Kock:ClsdCatsGenCommMnds}
A.~Kock, \emph{Closed categories generated by commutative monads}, J. Austral.
  Math. Soc. \textbf{12} (1971), 405--424.

\bibitem{Lip}
J.~Lipman, \emph{Notes on derived functors and {G}rothendieck duality},
  Foundations of {G}rothendieck duality for diagrams of schemes, Lecture Notes
  in Math., vol. 1960, Springer, Berlin, 2009, Available at
  \verb|http://www.math.purdue.edu/~lipman/Duality.pdf|, pp.~1--259.

\bibitem{Lu:PhD}
R.~B.~B. Lucyshyn-Wright, \emph{{R}iesz-{S}chwartz extensive quantities and
  vector-valued integration in closed categories}, Ph.D. thesis, York
  University, 2013, \href{http://arxiv.org/abs/1307.8088}{arXiv:1307.8088}.

\bibitem{McCr}
P.~McCrudden, \emph{Balanced coalgebroids}, Theory Appl. Categ. \textbf{7}
  (2000), No. 6, 71--147.

\bibitem{SchPr}
C.~J. Schommer-Pries, \emph{The classification of two-dimensional extended
  topological field theories}, 2009, Thesis (Ph.D.)--University of California,
  Berkeley; available as arXiv:1112.1000.

\bibitem{St:CatsPost}
R.~Street, Message to the \textit{categories} electronic mailing list, August
  21, 2003, Available at \verb|http://www.mta.ca/~cat-dist/archive/2003/03-8|.

\end{thebibliography}

\end{document}